\documentclass[journal]{IEEEtran}
\usepackage{tikz}
\usetikzlibrary{shapes, shapes.geometric, shapes.symbols, shapes.arrows, shapes.multipart, shapes.callouts, shapes.misc}
\usepackage[activate={true,nocompatibility},final,tracking=true,kerning=true,spacing=true,factor=1100,stretch=10,shrink=10]{microtype}

\usepackage{etex}
\usepackage{graphicx}
\usepackage{amsmath}
\usepackage{amssymb}
\usepackage{mathrsfs}
\usepackage[vlined,ruled]{algorithm2e}

\usepackage[font=small]{caption}
\usepackage{subcaption}

\usepackage{cite}
\usepackage{epstopdf}
\usepackage{booktabs}

\newtheorem{theorem}{Theorem}[section]
\newtheorem{lemma}[theorem]{Lemma}

\newtheorem{proposition}[theorem]{Proposition}
\newtheorem{problem}{Problem}
\newtheorem{remark}{Remark}

\newcounter{ass}[section]
\renewcommand*{\theass}{\thesection.\arabic{ass}}

\newcommand{\PQ}{\textsf{PQ} }
\newcommand{\PV}{\textsf{PV} }
\renewcommand{\j}{\boldsymbol{\mathrm{j}}}

\newenvironment{pfof}[1]{\vspace{1ex}\noindent{\itshape Proof of
    #1:}\hspace{0.5em}} {\hfill\oprocend\vspace{1ex}}
\newenvironment{proof}[1]{\vspace{1ex}\noindent{\itshape Proof:}\hspace{0.5em}} {\hfill\oprocend\vspace{1ex}}

\newcommand{\setdef}[2]{\{#1 \; | \; #2\}}

\newcommand{\map}[3]{#1: #2 \rightarrow #3}

\newcommand{\real}{\mathbb{R}}

\newcommand\oprocendsymbol{\hbox{$\square$}}
\newcommand\oprocend{\relax\ifmmode\else\unskip\hfill\fi\oprocendsymbol}

\newcommand{\vect}[1]{\mathbbold{#1}}
\newcommand{\vones}[1][]{\vect{1}_{#1}}
\newcommand{\vzeros}[1][]{\vect{0}_{#1}}
\DeclareSymbolFont{bbold}{U}{bbold}{m}{n}
\DeclareSymbolFontAlphabet{\mathbbold}{bbold}

\newcommand{\tb}{\color{black}}

\renewcommand{\circle}{\mathbb{S}^1}

\newcommand{\bsin}{\boldsymbol{\sin}}

\newcommand{\Ell}{\mathcal{E}^{\ell\ell}}
\newcommand{\Egl}{\mathcal{E}^{g\ell}}
\newcommand{\Egg}{\mathcal{E}^{gg}}

\usepackage{enumerate}

\usepackage{xcolor}

\ifCLASSINFOpdf
\else
\fi

\graphicspath{{./fig/}}

\begin{document}
%
\title{A Theory of Solvability for Lossless Power Flow Equations -- Part II: {\tb Conditions for Radial Networks}}
%
%
%
\author{John~W.~Simpson-Porco,~\IEEEmembership{Member,~IEEE}\\
\thanks{J.~W.~Simpson-Porco is with the Department of Electrical and Computer Engineering, University of Waterloo. Email: jwsimpson@uwaterloo.ca.}
}

%
%

\markboth{IEEE Transactions on Control of Network Systems. This version: \today}%
{Shell \MakeLowercase{\textit{et al.}}: Bare Demo of IEEEtran.cls for Journals}
%



\maketitle


\vspace{-2em}

\begin{abstract}
This two-part paper details a theory of solvability for the power flow equations in lossless power networks. 
In Part I, we derived a new formulation of the lossless power flow equations, which we term the fixed-point power flow. The model is parameterized by several graph-theoretic matrices -- the power network stiffness matrices -- which quantify the internal coupling strength of the network.
In Part II, we leverage the fixed-point power flow to study power flow solvability. For radial networks, we derive parametric conditions which guarantee the existence and uniqueness of a high-voltage power flow solution, and construct examples for which the conditions are also necessary. Our conditions (i) imply convergence of the fixed-point power flow iteration, (ii) unify and extend recent results on solvability of decoupled power flow, (iii) directly generalize the textbook two-bus system results, and (iv) provide new insights into how the structure and parameters of the grid influence power flow solvability.
\end{abstract}


\section{Introduction}
\label{Section: Introduction}


In the companion paper \cite{JWSP:17a}, we developed a new model of coupled power flow for lossless networks, which we termed the fixed-point power flow (FPPF). The name references the fact that for radial networks, the fixed-point power flow can be written as a fixed-point equation $v = f(v)$ in the scaled voltage magnitudes $v_i = V_i/V_i^*$ at \PQ buses. {\tb For meshed networks, the FPPF has \textemdash{} for each cycle \textemdash{} one additional slack variable $y_i$ and one additional nonlinear constraint, which ensure that Kirchhoff's voltage law is satisfied around cycles in the network. Phase angles are not present in the FPPF; instead, phase differences are recovered uniquely (modulo $2\pi$) from a solution $(v,y)$. We showed through numerical testing that iterating the FPPF provides an effective means of solving the power flow equations, converging linearly from a flat start in both lightly and heavily loaded networks, with nearly zero sensitivity to initialization.}
The reader is referred to Part I for a detailed introduction, motivation, modeling assumptions, and a complete derivation of the FPPF. We now shift our focus from modeling and computation to analysis, and address the following problem.

\smallskip

\begin{problem}[\bf Power Flow Solvability Problem]\label{Prob:Main}
Give necessary and/or sufficient conditions on the active and reactive power injections, the generator voltage magnitudes, the network topology, and the series/shunt admittances under which the power flow equations possess a unique, high-voltage solution. In addition, quantify the location of the solution in voltage-space and angle-space in terms of the problem data.
\end{problem}

\subsection{Contributions of Part II}

We leverage the FPPF developed in Part I, and for radial networks we derive sufficient conditions which ensure the power flow equations possess a solution. Our conditions guarantee the existence of a solution within a desirable set in voltage space, where voltage magnitudes are near their open-circuit values and phase angle differences between buses are small. For simplified topologies containing no connections between \PQ buses, we further show that our existence condition implies the fixed-point power flow mapping $f$ is a contraction mapping. This in turn implies that (i) the solution is unique within the specified set, and (ii) the FPPF iteration $v_{k+1} = f(v_k)$ converges exponentially to the unique power solution. For certain cases, we show our conditions are also necessary for existence. As a byproduct of our analysis, we also establish the \emph{non-existence} of solutions within both a ``medium-voltage'' and an ''extra-high-voltage'' region of voltage-space. The existence of this medium-voltage solutionless region implies a lower bound in voltage-space between the unique high-voltage solution and any undesirable low-voltage solutions.

The conditions we derive are parametric, depending only on the given data of the power flow problem including fixed active and reactive power injections, shunt and series susceptances, the network topology, and \PV bus voltage set-points. Rather than imposing spectral or worst-case bounds individually on these quantities, our conditions fuse the relevant parameters together into intuitive loading margins for the system, by exploiting the stiffness matrices introduced in Part I \cite[Def. 2]{JWSP:17a}. These loading margins generalize the textbook two-bus network feasibility results \cite[Chapter 2]{TVC-CV:98}, and unify recent results on feasibility of decoupled active \cite{FD-MC-FB:11v-pnas} and reactive \cite{jwsp-fd-fb:14c} power flow. While our results are currently restricted to an unrealistic class of networks (lossless, radial), the analysis presented is \textemdash{} in the authors opinion \textemdash{} the most complete one available in the literature. The theoretical results here also provide a partial explanation for the robust numerical behavior of the FPPF iteration observed in Part I.


\subsection{Organization of Paper}

Section \ref{Sec:Literature} presents an extensive literature review, surveying the history of power flow solvability results, with a focus on incorporating structural information into solvability conditions. 

Section \ref{Sec:TwoBus} briefly reviews the grid model, then presents a detailed analysis of the classic two-bus model of power flow between a \PQ bus and a \PV bus, providing context for our network generalizations. Section \ref{Sec:PartI} briefly re-states the required definitions and notation for the FPPF model.


In Section \ref{Sec:MainResults} we state and prove our main results on the existence and uniqueness of power flow solutions. The proof is technical; casual readers will wish to simply read the theorem statements along with the explanations which follow them. Section \ref{Sec:Conclusions} concludes and lists some open problems.

\subsection{Notation and Preliminaries}

We refer the reader to Part I \cite{JWSP:17a} for notational conventions.

\smallskip

\emph{Fixed Point Theory:} A point $x^* \in \real^n$ is a \emph{fixed point} of a map $\map{f}{\real^n}{\real^n}$ if $f(x^*) = x^*$. A set $\mathcal{X} \subset \real^n$ is \emph{invariant} for $f$ if $x \in \mathcal{X} \Rightarrow f(x) \in \mathcal{X}$. If $\mathcal{X}$ is a compact, convex, and invariant set for $f$, then $f$ possess at least one fixed point $x^* \in \mathcal{X}$ \cite[Section 7, Corollary 8]{EHS:94}. We say $f$ is a \emph{contraction map} on $\mathcal{X}$ if there exists a $\beta \in {[0,1)}$ (the \emph{contraction rate}) and a vector norm $\|\cdot\|$ such that for any $x,y \in \mathcal{X}$, $\|f(x)-f(y)\| \leq \beta\|x-y\|$\@. If $f$ is a contraction map on an invariant compact set $\mathcal{X}$, then $f$ has a unique fixed point $x^* \in \mathcal{X}$ \cite[Theorem 9.32]{WR:76}, and the iterates of $x_{k+1}=f(x_k)$ of $f$ from any initial point $x_0 \in \mathcal{X}$ satisfy $\|x_k-x^*\| \leq \beta^k\|x_0-x^*\|$.

\section{Literature on Power Flow Solvability}
\label{Sec:Literature}

\subsection{Summary of Results: 1972 -- 2016}

Korsak \cite{AJK:72} appears to be among the first to explicitly study the solution space of power flow equations and, in the context of swing stability, constructed an example showing that decoupled active power flow can possess multiple ``stable'' solutions in meshed networks. Simultaneously, Tavora and Smith \cite{CJT-OJMS:72b} studied those same equations, giving a necessary condition for solvability and studying the singular surfaces of the active power flow Jacobian for low-dimensional systems. A more detailed study of decoupled active power flow was later performed by Araposthatis, Sastry and Varaiya \cite{AA-SS-VP:81,AA-PV:83}.
Galiana and Jarjis \cite{FDG:75,JJ-FDG:81,FDG:83}  investigated the coupled equations in rectangular coordinates, proposing feasibility conditions based on supporting hyperplanes for a generalized injection region in parameter-space. Unfortunately the assumption that the injection region is a convex  cone was later shown to be incorrect; see, for example, \cite{YVM-DJH-IAH:00,IAH-RJD:01,YVM-DZY-DJH:08,YVM-BV-DW-BL-ZH-STE-ZHH:14}. Soon after, Baillieul and Byrnes \cite{JB-CIB:82} leveraged geometric techniques to provide a combinatorial upper bound on the number of complex solutions; for particularly clear expositions on the nature of multiple solutions, see the articles \cite{BKJ:77,AK-JW:91,DKM-DM-MN:16}.

Wu was the first to apply the fixed point and degree techniques developed for nonlinear circuits to power flow \cite{FFW:77}. In two breakthrough papers \cite{FFW-SK:80,FW-SK:82}, Wu and Kumagai applied the Leray-Schauder fixed point theorem to derive conditions for solvability of both the decoupled active and reactive power flow equations. Using homotopy and degree theory, they showed that under some additional conditions, the coupled power flow equations also possess a solution. 
In a similar spirit, Thorp, Schulz and Ili\'{c} \cite{JT-DS-MIS:86,MI:92} later applied nonlinear circuit results 
to derive a conservative condition for solvability of decoupled reactive power flow. In \cite{HDC-MEB:90} Chiang and Baran studied feasibility on distribution feeders using Baran's branch flow model, arguing in a quasi-quantitative manner that the system has a unique solution. Unfortunately the arguments provided no quantitative bounds. Grijalva and Sauer \cite{SG-PWS:05,SG:12} proposed necessary conditions for Jacobian singularity based on the saturation of transmission lines. In \cite{BCL-PWS-MAP:99} Lesieutre, Sauer, and Pai derived sufficient solvability conditions, but the analysis forbids constant power loads. {\tb In \cite{WD-AEB-RO-FLL:09} voltage magnitudes were isolated in terms of voltage phase angles, yielding a semi-explicit power flow solution, but the analysis is restricted to impedance load models.}

While little progress was made through the 2000s, a recent flurry of activity has yielded new results. In \cite{FD-MC-FB:11v-pnas,FD-FB:13c} D\"{o}rfler \emph{et al.} noted that the decoupled active power flow in polar form could be \emph{exactly} solved on acyclic networks, leading them to an explicit necessary and sufficient condition for existence and uniqueness. The condition was also shown to also be tight for networks with short cycles, and a conservative sufficient condition was given for general meshed networks; see also \cite{NA-SG:13}. {\tb Barabanov \emph{et al.} proposed a necessary condition for LTI systems with constant power loads \cite{NB-RO-RG-BP:16}.} Molzahn \emph{et al.} proposed a necessary condition for the coupled lossy power flow problem. In \cite{jwsp-fd-fb:14c} the author and his collaborators derived a sufficient condition for solvability of decoupled reactive power flow, analogous to the active power flow results in \cite{FD-MC-FB:11v-pnas}. For distribution feeders, Bolognani and Zampieri \cite{SB-SZ:15} reformulated the complex-valued power flow as a fixed point mapping, and derived a condition under which there exists a unique solution, with a relaxed condition proposed later by Yu \emph{et al.} in \cite{SY-HDN-KST:15}, and further distribution system results in \cite{ZW-BC-JW:16,CW-AB-JYLB-MP:16}.

As a complement to fixed point-based approaches, convexity has emerged as a powerful tool for analyzing power flow solvability. In the context of optimal power flow, sufficient conditions have been found for both radial and meshed networks under which various convex relaxations are exact; see \cite{SHL:14a,SHL:14b} for a broad survey of recent results. The implication is that under such conditions, convex programming can be used to iteratively find the solution, or otherwise certify that no solution exists. Inspired by classical work on transient stability, Dvijotham \emph{et al.} have pursued the energy-function formulation of lossless power flow and derived linear matrix inequality conditions under which the energy function is convex on a suitable subset of voltage-space \cite{KD-SL-MC:15a}. This work has subsequently been broadened to the study of the lossy power flow equations by leveraging monotone operator theory \cite{KD-SL-MC:15b}. 
Finally, we note that at times there is a blur between the literature on power flow solvability and the literature on quasi-static voltage stability and collapse. Surveying the latter is far beyond the scope of this article, but the interested reader can find many connections in \cite{TTL-RAS:91,MKP:93,CAC:02,CAC:95} and in the text \cite{TVC-CV:98}.

To summarize, the literature now contains a fairly robust understanding of existence and uniqueness of a desirable solution for both decoupled active and reactive power flow \cite{FD-MC-FB:11v-pnas,jwsp-fd-fb:14c}; the only remaining area of confusion is for decoupled active power flow in heavily meshed networks, {\tb where surprising examples continue to be found \cite{RD-TC-PJ:17}}. The situation for both lossless and lossy coupled power flow remains comparatively opaque, with only implicit, necessary, or restrictive results available. In Section \ref{Sec:MainResults} we present new results in this direction.


\subsection{Grid Topology and Power Flow Solvability}
\label{Sec:Structural}

One goal of this paper is to develop a useful theory for studying how topology and grid parameters influence power flow solvability. The perplexing role of network topology has long been noted by those closest to the feasibility problem. In his 1975 paper \cite{FDG:75}, Galiana pondered \emph{``\textrm{[Power flow feasibility]} is one question which is unresolved in power systems analysis, but which is of basic theoretical and practical importance \ldots is a given network structurally susceptible to unfeasibility? What type and what value of injections are most likely to result in unfeasible situations?''} 
Decades later, Ili\'{c} \cite{MI:92} noted that \emph{``A nonlinear network based formulation of the coupled real power/voltage problem is recognized here as an open research problem \ldots the information on network topology could significantly change conservativeness of the results.''} Later still, Hill and Chen \cite{DJH-GC:06} write \emph{``The power systems theory needs to be pushed further in the direction of exploiting structural features of the networks.''} As noted in \cite{PH-CS-SB:10} however, topological information by itself \emph{is not} particularly useful. What \emph{is} useful is the careful fusion of topological information with other parameters, most importantly the sizes and locations of injections. Our formulation accomplishes this by exploiting the stiffness matrices introduced in Part I.

\section{Grid Model and Solvability of the Two-Bus Lossless Power Flow Model}
\label{Sec:TwoBus}

\subsection{Network and AC Power Flow Model}
\label{Sec:Network}

We consider a synchronous AC power network in steady-state. Let $(\mathcal{N},\mathcal{E})$ be the graph describing the network. The set of buses $\mathcal{N}$ is partitioned as $\mathcal{N} = \mathcal{N}_L \cup \mathcal{N}_G$, into $n$ loads (\PQ buses), denoted by $\mathcal{N}_L$, and $m$ generators (\PV buses) denoted by $\mathcal{N}_G$. Each bus has a voltage magnitude $V_i > 0$, a voltage phase angle $\theta_i \in \circle$, and an active (resp. reactive) power injection $P_i \in \real$ ($Q_i \in \real$). At \PV buses $i \in \mathcal{N}_G$, $P_i$ and $V_i$ are fixed, while at \PQ buses $i \in \mathcal{N}_L$, $P_i$ and $Q_i$ are fixed; we assume throughout that $Q_i \leq 0$ for each $i \in \mathcal{N}_L$, which is the most common case of inductive loads.
The set of directed branches $\mathcal{E} \subseteq \mathcal{N} \times \mathcal{N}$ is partitioned accordingly as\footnote{With some abuse of notation, we will write $\vones[\ell\ell]$ for the vector of all ones of length $|\mathcal{E}^{\ell\ell}|$, and similarly for the other sets.}
\begin{equation}\label{Eq:Edges}
\mathcal{E} = \Ell \cup \Egl \cup \Egg\,.
\end{equation}
For example, $\Ell$ contains all branches between the \PQ buses $i,j \in \mathcal{N}_L$. Without loss of generality, all branches $(i,j)\in\mathcal{E}^{g\ell}$ between generators and loads are oriented from generators to loads; other branches are assigned arbitrary directions.
Each branch $(i,j)\in\mathcal{E}$ is purely inductive with a susceptance $b_{ij} < 0$, and the susceptance matrix $B \in \real^{(n+m)\times(n+m)}$ is defined component-wise as $B_{ij} = B_{ji} = -b_{ij}$ for $i \neq j$, and $B_{ii} = \sum_{j=1,j\neq i}^{n+m} b_{ij} + b_{\mathrm{shunt},i}$, where $b_{\mathrm{shunt},i}$ is the shunt susceptance at bus $i\in\mathcal{N}$. In vector notation, $V$, $Q$, and $B$ inherit the bus partitioning $\mathcal{N} = \mathcal{N}_L\cup\mathcal{N}_G$ as
\begin{equation}\label{Eq:Susceptance}
V = \begin{pmatrix} V_L \\ V_G \end{pmatrix}\,,\quad Q = \begin{pmatrix}Q_L \\ Q_G \end{pmatrix}\,,\quad B = \begin{pmatrix}
B_{LL} & B_{LG}\\
B_{GL} & B_{GG}
\end{pmatrix}\,.
\end{equation}
The balance of complex power at each bus leads the lossless power flow equations
\begin{subequations}
\begin{align}\label{Eq:Active}
P_i &= \sum_{j=1}^{n+m} \nolimits V_iV_jB_{ij}\sin(\theta_i-\theta_j)\,,\quad i \in \mathcal{N}_L \cup \mathcal{N}_G\,,\\\label{Eq:Reactive}
Q_i &= -\sum_{j=1}^{n+m} \nolimits V_iV_jB_{ij}\cos(\theta_i-\theta_j)\,,\quad i \in \mathcal{N}_L\,.
\end{align}
\end{subequations}
Reactive power injections $Q_i$ at \PV buses omitted, as they may be determined as ``outputs'' after \eqref{Eq:Active}--\eqref{Eq:Reactive} are solved for the unknowns $\theta = (\theta_1,\ldots,\theta_{n+m})^{\sf T}$ and $V_L = (V_1,\ldots,V_n)^{\sf T}$. Without loss of generality, we assume that $\sum_{i=1}^{n+m} P_i = 0$, which is simply the lossless balance of active power.

\subsection{The Incidence Matrix and Related Constructions}

We denote the incidence matrix of the graph by $A \in \real^{(n+m)\times|\mathcal{E}|}$. If the network is radial (contains no cycles), then $\mathrm{ker}(A) = \emptyset$. In this case, for every vector of active power injections $P = (P_1,\ldots,P_{n+m})^{\sf T}$ satisfying the balance $\sum_{i=1}^{n+m}P_i = 0$, it follows that
\begin{equation}\label{Eq:BranchFlows}
p = (p_{\ell\ell},p_{g\ell},p_{gg})^{\sf T} \triangleq (A^{\sf T}A)^{-1}A^{\sf T}P\,
\end{equation}
is the \emph{unique} vector of branch-wise active power flows, satisfying Kirchhoff's current law $P = Ap$. The branch flows $p$ of inherit the partitioning of the branches, and $A$ inherits the bus and branch partitions and may be written as the block matrix
\begin{equation}\label{Eq:Incidence}
A = \begin{pmatrix}A_L \\ A_G\end{pmatrix} = 
\begin{pmatrix}
A_L^{\ell\ell} & A_L^{g \ell} & \vzeros[]\\
\vzeros[] & A_G^{g \ell} & A_G^{g g}
\end{pmatrix}\,.
\end{equation}
Figure \ref{Fig:Notation} illustrates these conventions on a simple network.
\begin{figure}[ht!]
\begin{center}
\includegraphics[width=1\columnwidth]{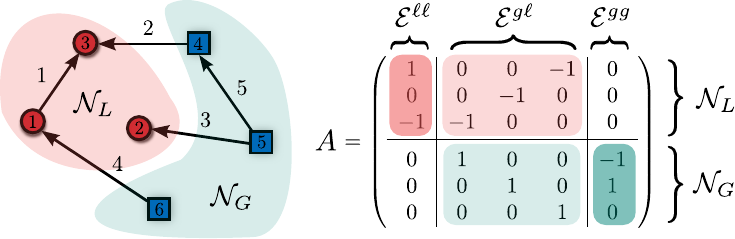}
\caption{Example network showing division of buses and edges with $|\mathcal{N}_G| = 3$ blue \PV buses, $|\mathcal{N}_L| = 3$ red \PQ buses, $|\Ell| = 1$, $|\Egl| = 3$, and $|\Egg| = 1$. Edges $(i,j) \in \mathcal{E}^{g\ell}$ are oriented from \PV buses to \PQ buses, while the orientation of other edges is arbitrary.}
\label{Fig:Notation}
\end{center}
\end{figure}
Since each element of $A$ is either $\pm 1$ or zero, we may write $A$ as the difference two binary matrices $A = A(+)-A(-)$, which have the analogous partitionings
\begin{equation}\label{Eq:IncidencePlusMinus}
A(\pm) = 
\begin{pmatrix}
A_L^{\ell\ell}(\pm) & A_L^{g\ell}(\pm) & \vzeros[]\\
\vzeros[] & A_G^{g \ell}(\pm) & A_G^{gg}(\pm)
\end{pmatrix}
\end{equation} 
The matrix $A(+)$ indexes the buses at the sending end of each branch, while $A(-)$ indexes the corresponding receiving end buses.\footnote{Due to our branch orientation conventions, $A_L^{g\ell}(+) = A_G^{g\ell}(-) = \vzeros[]$.} We also will use an ``unoriented'' version $|A|$ of the incidence matrix $A$, with all non-zero elements set to +1:
\begin{align}\label{Eq:IncidenceAbs}
|A| &= \begin{pmatrix}|A|_L \\ |A|_G\end{pmatrix} = 
\begin{pmatrix}
|A|_{L}^{\ell\ell} & |A|_L^{g\ell} & \vzeros[]\\
\vzeros[] & |A|_G^{g\ell} & |A|^{gg}_G
\end{pmatrix}\,.
\end{align}

\subsection{Solution of Two-Bus Model}
\label{Sec:SolutionTwoBus}

In this section we cover in detail the solution of the canonical two-bus power system model \cite[Chapter 2]{TVC-CV:98}. The insights we gain from a detailed analysis of this problem will help us properly understand our main results, as we will obtain analogous results for networks by analyzing the FPPF. The material here is known, but our treatment is quite non-standard.

The two-bus system depicted in Figure \ref{Fig:Single} consists of a single \PV bus connected to a single \PQ bus. The transmission line has a susceptance $-b < 0$ and the \PQ bus demands a constant complex power $-(P_L+\j Q_L)$.
\begin{figure}[h!]
\begin{center}
\includegraphics[width=0.8\columnwidth]{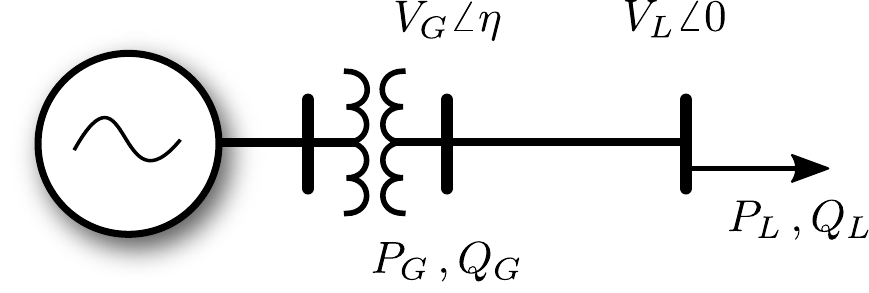}
\caption{Single-line diagram for the two-bus \PV\!\!-\PQ power network, where a generator feeds power to a stiff load through an inductive transmission line.}
\label{Fig:Single}
\end{center}
\end{figure}
In this case, the power flow equations \eqref{Eq:Active}--\eqref{Eq:Reactive} reduce to
\begin{subequations}\label{Eq:TwoNodePowerFlow}
\begin{align}
\label{Eq:SingleLine1}
P_L &= bV_GV_L\sin(-\eta)\\
\label{Eq:SingleLine2}
P_G &= bV_GV_L\sin(\eta)\\
\label{Eq:SingleLine3}
Q_L &= bV_L^2 - bV_LV_G\cos(\eta)\,,
\end{align}
\end{subequations}
where subscripts denote generation and load, respectively, and $\eta$ is the difference in phase angles between the two buses. 
We assume without loss of generality that $P_G = -P_L$; this is a necessary condition for \eqref{Eq:TwoNodePowerFlow} to be solvable (Section \ref{Sec:Network}). The incidence matrix is given by $A = \begin{pmatrix}-1 & 1\end{pmatrix}^{\sf T}$, and since the network is radial, Kirchhoff's current law $\begin{pmatrix}P_L & P_G\end{pmatrix}^{\sf T} = Ap$ has the unique solution $p = P_G$ for the branch active power flow $p$. It follows that $p$ satisfies 
\begin{equation}\label{Eq:SingleLine4}
p = bV_GV_L\sin(\eta)\,.
\end{equation}
We restrict our attention to \eqref{Eq:SingleLine3} and \eqref{Eq:SingleLine4}, which have unknowns $V_L$ and $\eta$. Note that under open-circuit conditions, where $P_G = P_L = Q_L = 0$, inspection yields that the unique open-circuit solution $(\eta^*,V_L^*) \in [-\frac{\pi}{2},\frac{\pi}{2}] \times \real_{>0}$ to \eqref{Eq:SingleLine3}--\eqref{Eq:SingleLine4} is $V_L^* = V_G$ and $\eta^* = 0$.
The first step in solving \eqref{Eq:SingleLine3}--\eqref{Eq:SingleLine4} is to eliminate $\eta$ by squaring both sides of both equations, adding, and using $\sin^2(\eta)+\cos^2(\eta) = 1$ to arrive at the quartic equation
$$
p^2 + (Q_L-bV_L^2)^2 = b^2V_G^2V_L^2\,,
$$
or equivalently
\begin{equation}\label{Eq:QuarticVoltage}
(bV_L^2)^2 - (2Q_Lb+b^2V_G^2)V_L^2 + p^2 + Q_L^2 = 0\,.
\end{equation}
Now introduce the change of variables $v = V_L/V_L^*$, to the scaled voltage variable $v >0$. Inserting this into \eqref{Eq:QuarticVoltage}, collecting terms, and dividing through by constants, one obtains
\begin{equation}\label{Eq:SingleLineIntermediate}
v^4 - v^2\left(1-\frac{\Delta}{2}\right) + \frac{1}{16}\Delta^2 + \Gamma^2 = 0\,,
\end{equation}
where we have introduced the dimensionless active and reactive power variables $\Gamma, \Delta \in \real$ defined by\footnote{While the sign of $\Gamma$ ends up being unimportant, $\Delta$ has been defined such that $\Delta > 0$ corresponds to an inductively loaded network.} 
\begin{equation}\label{Eq:SingleLineParameters}
\Gamma \triangleq \frac{p}{b(V_L^*)^2}\,,\qquad
\Delta \triangleq \frac{Q_L}{-\frac{1}{4}b(V_L^*)^2}\,.
\end{equation}
The nondimensionalized equation \eqref{Eq:SingleLineIntermediate} is quadratic in $v^2$, with solutions
\begin{equation}\label{Eq:SingleLineIntermediate2}
v_{\pm}^2 = \frac{1}{2}\left(1-\frac{\Delta}{2}\pm \sqrt{1-(4\Gamma^2+\Delta)}\right)\,.
\end{equation}
Both solutions take nonnegative real values if and only if 
\begin{equation}\label{Eq:SingleLineCondition}
\boxed{
\Delta + 4\Gamma^2 < 1\,,
}
\end{equation}
in which case the solutions
\begin{equation}\label{Eq:deltasingleline}
v_{\pm} = \sqrt{\frac{1}{2}\left(1-\frac{\Delta}{2}\pm\sqrt{1-\left(4\Gamma^2+\Delta\right)}\right)}\,
\end{equation}
satisfy $v_- \in {[0,\frac{1}{\sqrt{2}})}$ and $v_+ \in {(\frac{1}{2},1]}$, with ordering $v_- < v_+$. Under condition \eqref{Eq:SingleLineCondition} then, it follows that \eqref{Eq:SingleLineIntermediate} possesses a unique solution $v = v_+$ in the high-voltage set $\setdef{v \in \real}{v_+ \leq v \leq 1}$, and in this case the solution lives at the lower boundary of this set. This situation is depicted in Figure \ref{Fig:delta_minus_plus}.

\begin{figure}[!ht]
        \begin{center}
        \begin{subfigure}[!ht]{0.9\columnwidth}
        		\centering
                \includegraphics[height=0.45\columnwidth]{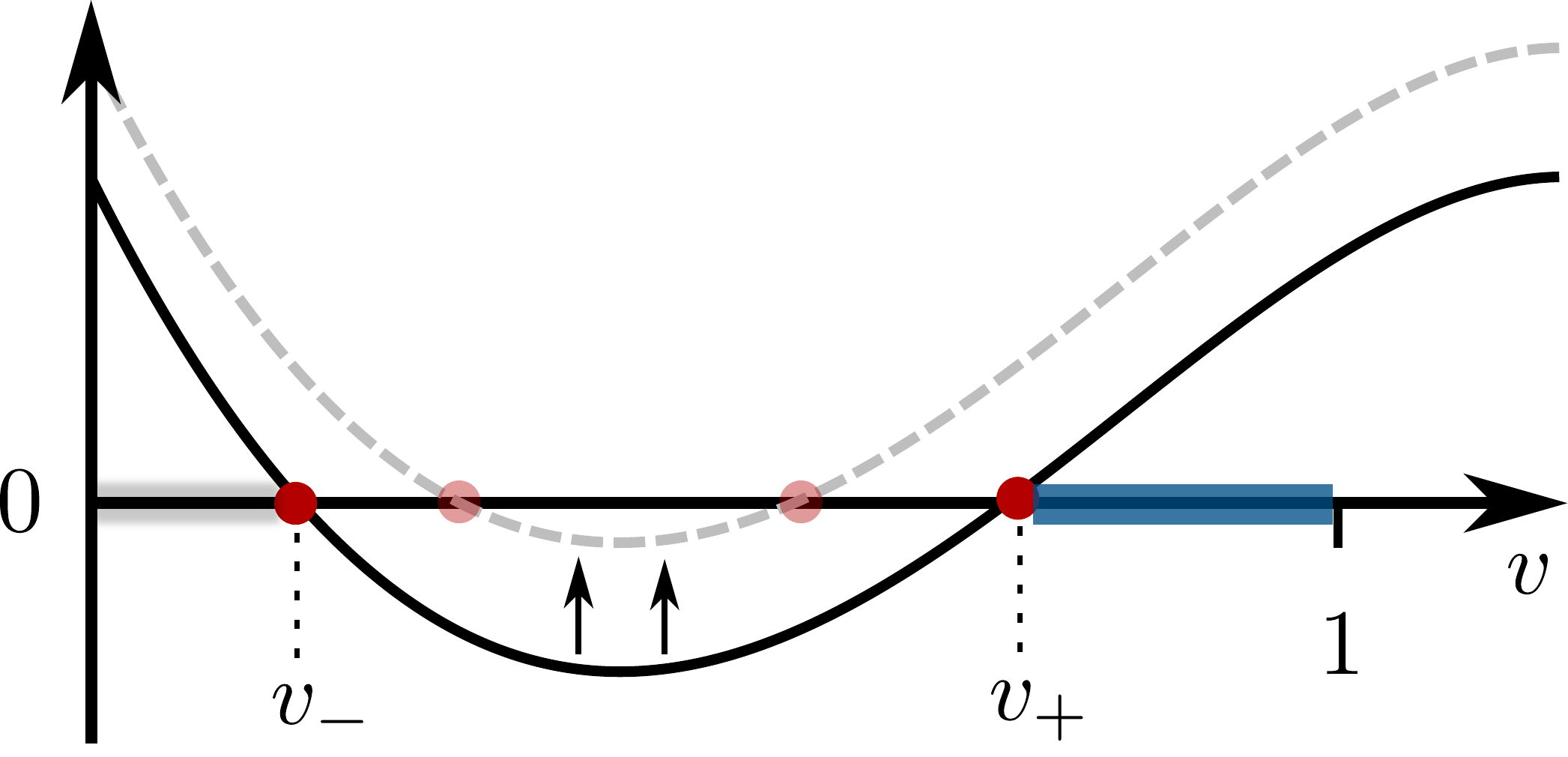}
                \caption{Depiction of curve given by the left-hand side of \eqref{Eq:SingleLineIntermediate}, with two solutions $v_-$ and $v_+$. As either $\Delta$ or $\Gamma$ is increased, the curve shifts up, and the solutions move towards one another.}
                \label{Fig:delta_minus_plus}
        \end{subfigure}
        
        \vspace{1em}
        \begin{subfigure}[!ht]{0.9\columnwidth}
        		\centering
                \includegraphics[height=0.45\columnwidth]{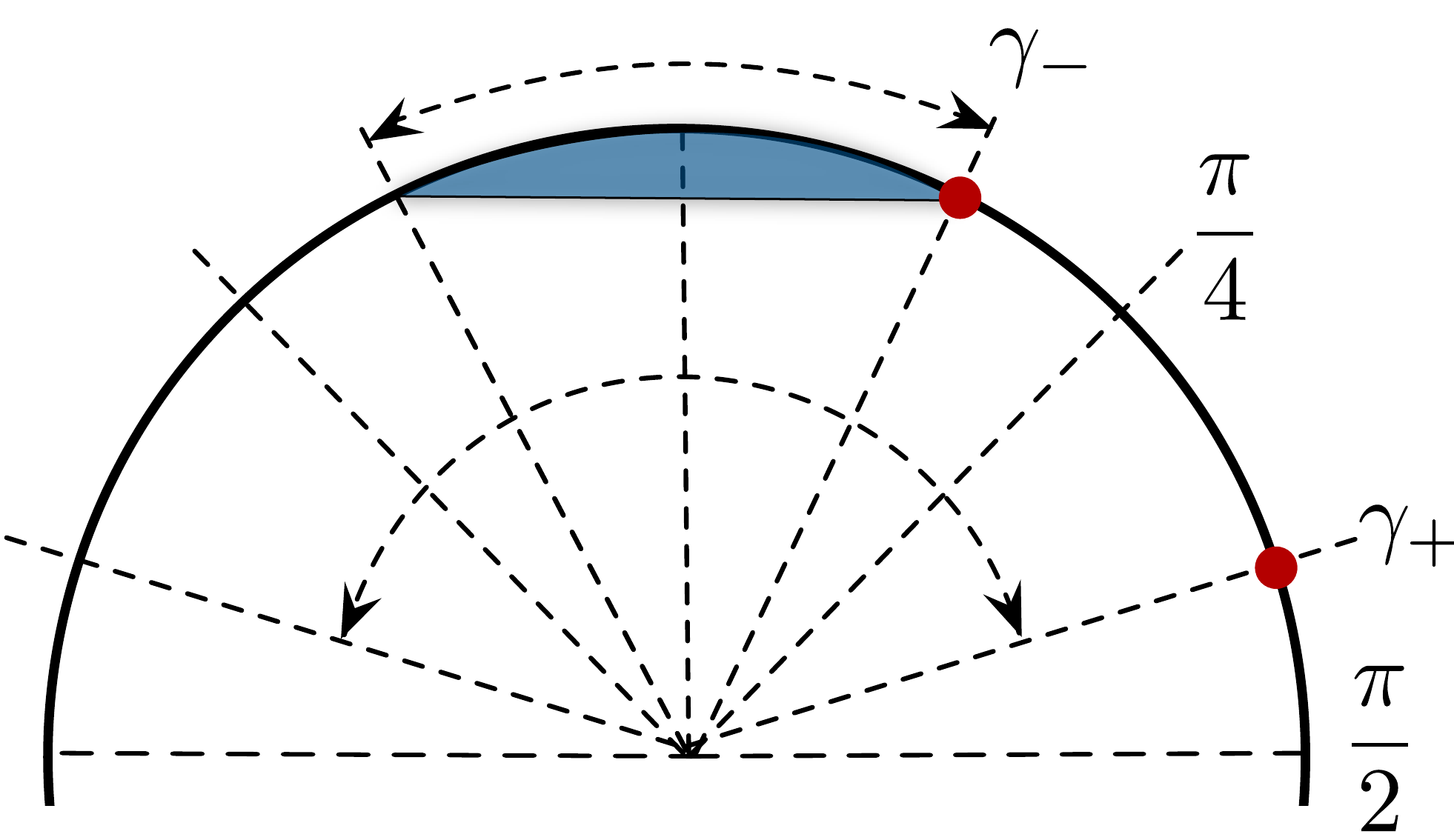}
                \caption{Illustration of solution space for angle difference $\eta$, with two solutions $\gamma_-$ and $\gamma_+$.}
                \label{Fig:gamma_minus_plus}
                \end{subfigure}

        \caption{Illustrations of voltage and angle solution spaces for two-bus \PV\!\!--\PQ power flow model under the solvability condition \eqref{Eq:SingleLineCondition}; (b) is adapted from \cite{FD-FB:12i}. Note that both $v_+$ and $\gamma_-$ are the unique solutions contained in the closure of the blue shaded regions.}
\end{center}
\end{figure}

Having determined the load voltage, we can determine the phase angle difference $\eta$ by rearranging \eqref{Eq:SingleLine4} to read as
$$
\sin(\eta) = \Gamma/v\,.
$$
Solutions are given by $\eta_{\pm} = \mathrm{sign}(p)\, \gamma_{\pm}$, where
\begin{equation}\label{Eq:gammasingleline}
\sin(\gamma_{\pm}) \triangleq \frac{|\Gamma|}{v_{\mp}}\,.
\end{equation}
These solutions satisfy $\gamma_- \in {[0,\frac{\pi}{4})}$ and $\gamma_+ \in {[0,\frac{\pi}{2})}$, with the ordering $\gamma_- < \gamma_+$. In particular then, \eqref{Eq:SingleLine4} possess a unique solution $\eta = \gamma_-$ in the small angle deviation set $\setdef{\eta \in [-\frac{\pi}{2},\frac{\pi}{2}]}{|\eta| \leq \gamma_-}$, and in this case the solution lives on the boundary of this set. This situation is depicted in Figure \ref{Fig:gamma_minus_plus}. A weak version of these results is stated formally now.

\medskip

\begin{proposition}[\bf Two-Bus Power Flow]\label{Prop:SingleLine}
Consider the two-bus power flow model \eqref{Eq:SingleLine1}--\eqref{Eq:SingleLine3} and define the dimensionless active and reactive loading margins $\Gamma$ and $\Delta$ as in \eqref{Eq:SingleLineParameters}. If the voltage stability condition \eqref{Eq:SingleLineCondition} holds, let $v_+ \in (\frac{1}{2},1]$, $\gamma_- \in {[0,\frac{\pi}{4})}$, and $v_- \in {[0,v_+)}$ be defined as in \eqref{Eq:deltasingleline} and \eqref{Eq:gammasingleline}. Then the power flow equations \eqref{Eq:SingleLine1}--\eqref{Eq:SingleLine3}
\begin{enumerate}
\item[(i)]  possess a unique solution $(\eta,V_L) \in [-\frac{\pi}{4},\frac{\pi}{4}] \times \real_{>0}$ with \PQ bus voltage $V_L$ satisfying $V_L/V_L^* \geq v_+$ and angle difference $\eta$ satisfying $|\eta| \leq \gamma_-$

\item[(ii)] possess no solutions satisfying $v_- < V_L/V_L^* < v_+$ and no solutions satisfying $V_L/V_L^* > 1$\,.
\end{enumerate}
\end{proposition}

\medskip

The dimensionless and loading margins $\Gamma$ and $\Delta$ in \eqref{Eq:SingleLineParameters} quantify the stress the network experiences under active and reactive loading, respectively. The numerator of each loading margin is the power demand, while the denominator of each quantifies the ``stiffness'' of the transmission network, in units of power. Roughly speaking then, the solvability condition \eqref{Eq:SingleLineCondition} requires that the network should be strong enough to handle the power transfer demands placed upon it. When \eqref{Eq:SingleLineCondition} holds, we are guaranteed the existence of exactly one solution with a load voltage no lower than $v_+$ and a phase angle difference no larger than $\gamma_-$. The main result of this paper is a generalization of Proposition \ref{Prop:SingleLine} to large networks. Moreover, as visible from Figure \ref{Fig:delta_minus_plus}, there is a medium-voltage region in voltage-space $\setdef{v \in \real}{v_- < v < v_+}$ that is devoid of solutions.

\smallskip

\section{The Fixed-Point Power Flow}
\label{Sec:PartI}

The fixed-point power flow is the main tool we will use to generalize the results of Proposition \ref{Prop:SingleLine}. For a detailed derivation of the fixed point power flow model, the reader is referred to Part I; here we simply state the model.

\subsection{The Power Network Stiffness Matrices}
\label{Sec:Stiffness}

Assuming the block $B_{LL}$ in \eqref{Eq:Susceptance} is negative definite, we define the \emph{open-circuit load voltages} $V_L^* \in \real^{n}_{>0}$ by
\begin{equation}\label{Eq:VLstar}
V_L^* \triangleq -B_{LL}^{-1}B_{LG}V_G\,.
\end{equation}
One may verify that $(\theta,V_L) = (\vzeros[n+m],V_L^*)$ is a solution of \eqref{Eq:Active}--\eqref{Eq:Reactive} with strictly positive voltage magnitudes when $P = \vzeros[n+m]$ and $Q_L = \vzeros[n]$. Using these open-circuit voltages, we define the scaled voltage variables $v = [V_L^*]^{-1}V_L$ and define the $|\mathcal{E}|\times|\mathcal{E}|$ \emph{branch stiffness matrix} $\mathsf{D}$ as the diagonal matrix\footnote{With a slight abuse of notation, in \eqref{Eq:VLstar} we also use $V_i^*$ or $V_j^*$ to denote any fixed \PV bus voltage.}
\begin{equation}\label{Eq:DMatrix}
\mathsf{D} \triangleq [V_i^*V_j^*B_{ij}]_{(i,j)\in\mathcal{E}},\,\, \mathsf{D} = \mathrm{blkdiag}(\mathsf{D}_{\ell\ell},\mathsf{D}_{g\ell},\mathsf{D}_{gg})\,.
\end{equation}
The branch stiffness matrix quantifies the strength of the transmission lines. Finally, using $B_{LL}$ and $V_L^*$ we define the $n \times n$ negative definite \emph{nodal stiffness matrix}\footnote{For a vector $x \in \real^n$, $[x] = \mathrm{diag}(x)$ is the associated diagonal matrix, $\bsin(x) = (\sin(x_1),\ldots,\sin(x_n))^{\sf T}$, and if $x$ is nonnegative, $\sqrt{x} = (\sqrt{x_1},\ldots,\sqrt{x_n})^{\sf T}$. The vector $\vones[n]$ is the vector of all ones of dimension $n$.}
\begin{equation}\label{Eq:Qcrit}
\mathsf{S} \triangleq \frac{1}{4}[V_L^*]B_{LL}[V_L^*]\,.
\end{equation}
The nodal stiffness matrix quantifies the strength of the transmission network between \PQ buses. The stiffness matrices \eqref{Eq:DMatrix} and \eqref{Eq:Qcrit} directly generalize the denominators in \eqref{Eq:SingleLineParameters}.

\subsection{Fixed Point Power Flow Model for Radial Networks}
\label{Sec:FPPF}

To compress our notation, define the function $\map{h}{\real^n_{>0}}{\real^{|\mathcal{E}|}}$ component-wise for any edge $e \sim (i,j) \in \mathcal{E}$ by
\begin{equation}\label{Eq:VComponentForm}
h_{e}(v) =
  \begin{cases}
   v_iv_j & \text{if } e\sim(i,j) \in \mathcal{E}^{\ell \ell}\,\\
   v_j       & \text{if } e\sim(i,j) \in \mathcal{E}^{g \ell}\,\\
   1 & \text{if } e\sim(i,j) \in \mathcal{E}^{gg}\,
  \end{cases}\,.
\end{equation}
In vector form, $h(v) \in \real^{|\mathcal{E}|}$ inherits the branch partitioning 
$$
h(v) = (h_{\ell\ell}(v),h_{g\ell}(v),h_{gg}(v))^{\sf T}
$$ 
of the branches. The following result is \cite[Corollary 3.7]{JWSP:17a}.

\smallskip

\begin{theorem}\label{Thm:FPPFAcyclic}\textbf{(Fixed Point Power Flow for Radial Networks)}
Consider the coupled flow equations \eqref{Eq:Active}--\eqref{Eq:Reactive} and assume that the graph $(\mathcal{N},\mathcal{E})$ describing the network is radial. The following two statements are equivalent:
\begin{enumerate}\setlength{\itemsep}{1.5pt}
\item[(i)] $(\theta,V_L) \in \Theta(\pi/2) \times \real^{n}_{>0}$ is a solution of \eqref{Eq:Active}--\eqref{Eq:Reactive};
\item[(ii)] $v \in \real^n_{>0}$ is a fixed point of the mapping $\map{f}{\real_{>0}^n}{\real^n}$ defined by
\begin{equation}\label{Eq:FixedPoint}
\begin{aligned}
f(v) &\triangleq \vones[n] -\frac{1}{4}\mathsf{S}^{-1}[Q_L][v]^{-1}\vones[n] + \frac{1}{4}\mathsf{S}^{-1}|A|_L^{g\ell}\mathsf{D}_{g\ell}u_{g\ell}(v)\\
&\quad + \frac{1}{4}\mathsf{S}^{-1}A_L^{\ell\ell}(+)\,[A_L^{\ell\ell}(-)^{\sf T}v]\mathsf{D}_{\ell\ell}u_{\ell\ell}(v)\\
&\quad + \frac{1}{4}\mathsf{S}^{-1}A_L^{\ell\ell}(-)\,[A_L^{\ell\ell}(+)^{\sf T}v]\mathsf{D}_{\ell\ell}u_{\ell\ell}(v)\,,
\end{aligned}
\end{equation}
where
\begin{subequations}
\begin{align}
\label{Eq:ull}
u_{\ell\ell}(v) &= \vones[\ell\ell] - \sqrt{\vones[\ell\ell] - [h_{\ell\ell}(v)]^{-2}\mathsf{D}_{\ell\ell}^{-2}[p_{\ell\ell}]p_{\ell\ell}}\,,\\
\label{Eq:ugl}
u_{g\ell}(v) &= \vones[g\ell] - \sqrt{\vones[g\ell] - [h_{g\ell}(v)]^{-2}\mathsf{D}_{g\ell}^{-2}[p_{g\ell}]p_{g\ell}}\,,
\end{align}
\end{subequations}
and $p_{\ell\ell},p_{g\ell}$ are the branch flows given by KCL in \eqref{Eq:BranchFlows}. The phase angle differences $\eta = A^{\sf T}\theta \in [-\frac{\pi}{2},\frac{\pi}{2}]^{|\mathcal{E}|}$ are determined uniquely by 
\begin{equation}
\label{Eq:psix}
\bsin(\eta) = [h(v)]^{-1}\mathsf{D}^{-1}p\,.
\end{equation}
\end{enumerate}
\end{theorem}
In other words, for radial networks the model \eqref{Eq:FixedPoint}--\eqref{Eq:psix} is equivalent to the model \eqref{Eq:Active}--\eqref{Eq:Reactive}; see \cite[Theorem 5.5]{JWSP:17a} for the more general meshed case.

\section{Main Results: Existence and Uniqueness of a Power Flow Solution}
\label{Sec:MainResults}

We now present our main results, giving sufficient and tight conditions for power flow solvability. In Section \ref{Sec:MainResults1} we treat the restricted case where there are no branches in the network between \PQ buses. In this case we can state quite strong results: we establish conditions which guarantee the existence of a solution in a specified set, its uniqueness within that set, necessity of the conditions, and the existence of a medium-voltage regime devoid of solutions. These results completely generalize the results for the two-bus model (Section \ref{Sec:SolutionTwoBus}).

In Section \ref{Sec:MainResults2}, we allow for general radial networks, including branches between \PQ buses; note that such branches are absent in the two-bus model. The analysis in this case proves to be much more challenging, and we give only conservative sufficient conditions for solution existence.

\subsection{Main Results 1: No Branches Between \PQ Buses}
\label{Sec:MainResults1}

We first restrict ourselves to networks where \PQ buses are not connected to one another; two example networks are shown in Figure \ref{Fig:example-network-1}. We allow for \PQ buses to have multiple \PV bus neighbors, as in Figure \ref{Fig:example-network-1a}, and denote by {\PV\!\!$(i) \subset \mathcal{N}_G$} the set of \PV buses connected to the $i$th \PQ bus:
$$
\PV(i) \triangleq \setdef{k \in \mathcal{N}_G}{(k,i)\in\mathcal{E}^{g\ell}}\,.
$$
There are no restrictions on the connections between \PV buses, other than that the network be radial.
\begin{figure}[!ht]
        \begin{center}
        \begin{subfigure}[!ht]{0.45\columnwidth}
        		\centering
                \includegraphics[height=0.35\columnwidth]{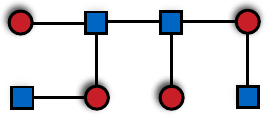}
                \caption{}
                \label{Fig:example-network-1a}
        \end{subfigure}~        
        \begin{subfigure}[!ht]{0.45\columnwidth}
        		\centering
                \includegraphics[height=0.35\columnwidth]{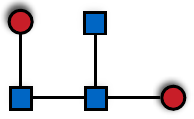}
                \caption{}
                \label{Fig:example-network-1b}
                \end{subfigure}
		\caption[]{Radial networks with no connections between \PQ buses. In (a), each \PQ bus ({\tikz\draw[black,semithick,fill=red] (0,0) circle (.5ex);}) 
		is allowed to have multiple \PV bus ({\tikz{\path[draw=black,semithick,fill={rgb:red,0;green,102;blue,196}] (0,0) rectangle (0.12cm,0.12cm);}}) neighbors. In (b), each \PQ bus has only one \PV bus neighbor.}
		\label{Fig:example-network-1}
\end{center}
\end{figure}
Under this assumption, $\mathcal{E}^{\ell\ell} = \emptyset$, and the final two terms in the FPPF \eqref{Eq:FixedPoint} are discarded. 
In this case, the grounded susceptance matrix $B_{LL}$ in \eqref{Eq:Susceptance} is a diagonal matrix, and hence so is the nodal stiffness matrix $\mathsf{S}$ in \eqref{Eq:Qcrit}, with strictly negative diagonal elements
\begin{align*}
\mathsf{S}_{ii} &= B_{ii}(V_i^*)^2 /4\,,
\end{align*}
where $B_{ii} = -\sum_{j\in\PV(i)} \nolimits B_{ij} + b_{\mathrm{shunt},i}$. Therefore, $\mathsf{S}_{ii}$ measures how strongly \PQ bus $i \in \mathcal{N}_L$ is connected to the neighboring generator buses $\PV(i)$.

\smallskip

\begin{theorem}\label{Thm:Main1}\textbf{(Sufficient Conditions for Solvability of Lossless Power Flow on Radial Networks I)}
Consider the lossless power flow equations \eqref{Eq:Active}--\eqref{Eq:Reactive}. Assume the network $(\mathcal{N},\mathcal{E})$ is radial, and that there are no branches between \PQ buses. Let the branch and nodal stiffness matrices $\mathsf{D}$ and $\mathsf{S}$ be as in \eqref{Eq:DMatrix} and \eqref{Eq:Qcrit} respectively, with the branch-wise active power flows {\tb $p = A^{\dagger}P$} as in \eqref{Eq:BranchFlows}.
Define the nodal and branch stress measures
\begin{subequations}
\begin{align}\label{Eq:Delta1}
\Delta_i &\triangleq Q_i/\mathsf{S}_{ii}\,, &\quad i &\in \mathcal{N}_L\\
\label{Eq:Gamma1gl}
\Gamma_i &\triangleq \max_{j\in\PV(i)}\nolimits |p_{ji}|/\mathsf{D}_{ji}\,, &\quad i &\in \mathcal{N}_L\\
\label{Eq:Gamma1gg}
\Gamma_{ij} &\triangleq |p_{ij}|/\mathsf{D}_{ij}\,, &\quad (i,j) &\in \mathcal{E}^{gg}\,.
\end{align}
\end{subequations}
If the above quantities satisfy
\begin{subequations}\label{Eq:Conditions1}
\begin{align}
\label{Eq:Omega1}
\max_{i\in\mathcal{N}_L}\,\,\, \Delta_i + 4\Gamma_{i}^2 &< 1\,,\\
\label{Eq:Gamma1}
\max_{(i,j)\in\mathcal{E}^{gg}} \Gamma_{ij} &< 1\,,
\end{align}
\end{subequations}
then the power flow equations \eqref{Eq:Active}--\eqref{Eq:Reactive} possess a unique  solution $(\theta,V_L)$, with \PQ bus voltages $V_i$ and power angles $\eta_{ij} = \theta_i-\theta_j$ satisfying the bounds
\begin{subequations}\label{Eq:SolutionBoundsMain1}
\begin{align}
{\tb v_{i,+} \leq V_i/V_i^*} &{\tb \leq 1}\,, &\quad i &\in \mathcal{N}_L\\
|\eta_{ji}| &\leq \gamma_i\,, &\quad (j,i) &\in \mathcal{E}^{g\ell}\\
|\eta_{ij}| &\leq \gamma_{ij}\,, &\quad (i,j) &\in \mathcal{E}^{gg}
\end{align}
\end{subequations}
where $v_{i,+} \in (\frac{1}{2},1]$, $v_{i,-} \in [0,\frac{1}{\sqrt{2}})$, $\gamma_i \in [0,\frac{\pi}{4})$ and $\gamma_{ij} \in [0,\frac{\pi}{2})$ are defined by
\begin{subequations}\label{Eq:DefinedQuantities}
\begin{align}
\label{Eq:delta1}
v_{i,\pm} &\triangleq \sqrt{\frac{1}{2}\left(1-\frac{\Delta_i}{2} \pm \sqrt{1-(\Delta_i+4\Gamma_{i}^2)}\right)}\,,\\
\label{Eq:gamma1}
\sin(\gamma_{i}) &\triangleq \frac{\Gamma_{i}}{v_{i,+}}\,,\quad \sin(\gamma_{ij}) \triangleq \Gamma_{ij}\,.
\end{align}
\end{subequations}
Moreover, the following statements hold:
\begin{enumerate}\setlength{\itemsep}{1.5pt}
\item[1)] \textbf{{No Medium-Voltage or Extra High-Voltage Solutions:}} There exist no solutions to \eqref{Eq:Active}--\eqref{Eq:Reactive} with voltage magnitudes $V_i$ satisfying
\begin{align*}
v_{i,-} &< V_i/V_i^* < v_{i,+}\,\qquad \text{or}\,\qquad V_i/V_i^* > 1\,
\end{align*}
for any $i \in \mathcal{N}_L$\,.
\item[2)] \textbf{{Necessity of Conditions:}} If $|\PV(i)| = 1$ for each $i \in \mathcal{N}_L$, meaning every \PQ bus has exactly one neighboring \PV bus, then the conditions \eqref{Eq:Omega1}--\eqref{Eq:Gamma1} are necessary and sufficient. Alternatively, if the network is loaded according to any of the loading profiles
\begin{enumerate}\setlength{\itemsep}{1.5pt}
\item[(i)] $Q_L = \alpha \mathsf{S}\vones[n]$ and $P = \vzeros[n+m]$
\item[(ii)] $Q_L = \vzeros[n]$ and $P = \frac{\alpha}{2} \,A \cdot (\mathsf{D}_{g\ell}\vones[g\ell],\vzeros[gg])$
\item[(iii)] $Q_L = \vzeros[n]$ and $P = \alpha\, A\cdot (\vzeros[g\ell],\mathsf{D}_{gg}\vones[gg])$
\end{enumerate}
where $\alpha \in [0,1)$, then the conditions \eqref{Eq:Omega1}--\eqref{Eq:Gamma1} are necessary and sufficient as functions of $\alpha$. In either case, the unique high-voltage solution satisfies the bounds \eqref{Eq:SolutionBoundsMain1} with equality sign.
\end{enumerate}
\end{theorem}

The statement of Theorem \ref{Thm:Main1} can feel intimidating, so we now spend some time walking the reader through it. The branch and nodal stiffness matrices $\mathsf{D}$ and $\mathsf{S}$ quantify the strength of the grid, in units of power.
In the case considered here, $\mathsf{S}$ is diagonal, and $\Delta_i = Q_i/\mathsf{S}_{ii}$ is interpreted as the stress experienced by \PQ bus $i \in \mathcal{N}_L$ due to reactive loading. Similarly, the ratio $\Gamma_{ji} = |p_{ji}|/\mathsf{D}_{ji}$ is the stress experienced by branch $(j,i)$ due to active power flow. Therefore $\Gamma_i$ in \eqref{Eq:Gamma1gl} identifies the most stressed branch incident to \PQ bus $i$. The quantities \eqref{Eq:Delta1}--\eqref{Eq:Gamma1gg} generalize the two-bus definitions \eqref{Eq:SingleLineParameters}.

The conditions \eqref{Eq:Conditions1} are ``low stress'' conditions; if the stress is not too great, the network has a high-voltage solution. Roughly speaking, \eqref{Eq:Omega1} is a voltage stability condition for \PQ buses, and limits the maximum downward pressure on \PQ bus voltage magnitudes; this is a generalization of the two-bus condition \eqref{Eq:SingleLineCondition}, and the set for a particular bus $i$ is depicted in Figure \ref{Fig:P1set}. Analogously, \eqref{Eq:Gamma1} is an angle stability condition for \PV buses, which limits the power flowing between generators. No such condition appeared for the two-bus case, as the two-bus case does not have multiple \PV buses. Together, the conditions \eqref{Eq:Conditions1} implicitly define a set of feasible power injections in the space of power variables. {\tb The ``bus-by-bus'' nature of the condition \eqref{Eq:Omega1} is not an accident. Indeed, the key to the proof is that for this class of networks, the FPPF decouples into $n$ independent scalar fixed-point equations, each of which can be analyzed independently.}

\begin{figure}[h!]
\begin{center}
\includegraphics[width=0.6\columnwidth]{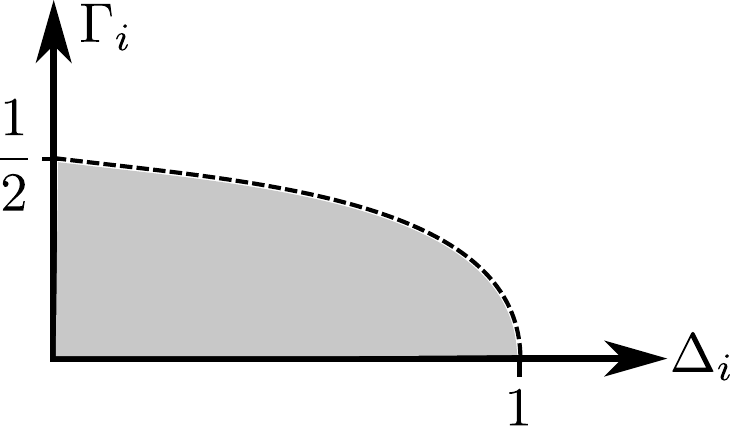}
\caption{Depiction of allowable set specified by \eqref{Eq:Omega1}.}
\label{Fig:P1set}
\end{center}
\end{figure}

The quantities $v_{i,+}$, $\gamma_{i}$, and $\gamma_{ij}$ give us extremely specific information on where this unique solution is in voltage and angle space. Every normalized \PQ bus voltage $v_i = V_i/V_i^*$ is guaranteed to be above $v_{i,+}$, while the phase angle difference $|\eta_{ji}|$ between \PQ bus $i\in\mathcal{N}_L$ to any neighboring \PV bus $j \in \mathcal{N}_G$ is guaranteed to be less than $\gamma_i < \frac{\pi}{4}$. Similarly, $\gamma_{ij} < \frac{\pi}{2}$ bounds the power angles between connected \PV buses $i,j \in \mathcal{N}_G$. Note that while angles between \PV buses may reach as high as 90$^{\circ}$, angles between \PV and \PQ buses are \emph{always} less than 45$^{\circ}$. The expression \eqref{Eq:delta1} shows precisely how active and reactive power injections (quantified through $\Delta_i$ and $\Gamma_i$) influence voltage magnitudes. For example, \eqref{Eq:delta1} could be used to generate estimates for PV and PQ curves.

Graphically, the results of Theorem \ref{Thm:Main1} are shown in Figure \ref{Fig:v-diagram}, where we plot the space of \PQ bus voltage magnitudes for the network in Figure \ref{Fig:example-network-1b}. Voltage-space is partitioned into three useful sets (i) a blue high-voltage set where the unique solution is found (ii) a red low-voltage set where undesirable solutions may or may not exist, and (iii) a grey medium-voltage and extra-high voltage set, which is devoid of solutions. The width of the grey medium-voltage set characterizes the ``stability margin'' of the system. In fact, a calculation using \eqref{Eq:DefinedQuantities} shows that
$$
v_{i,+}^2-v_{i,-}^2 = \sqrt{1-(\Delta_i+4\Gamma_i^2)}\,.
$$
Therefore, the margin with which the condition \eqref{Eq:Omega1} is satisfied is directly related to the proximity between the high-voltage solution and any low-voltage solution. In any of the described cases where the conditions \eqref{Eq:Conditions1} are also necessary, all solutions depicted in Figure \ref{Fig:v-diagram} are located at the ``corners'' of their respective boxes, nearest the centre of the figure. 

\begin{figure}[!ht]
        \begin{center}
                \includegraphics[width=0.7\columnwidth]{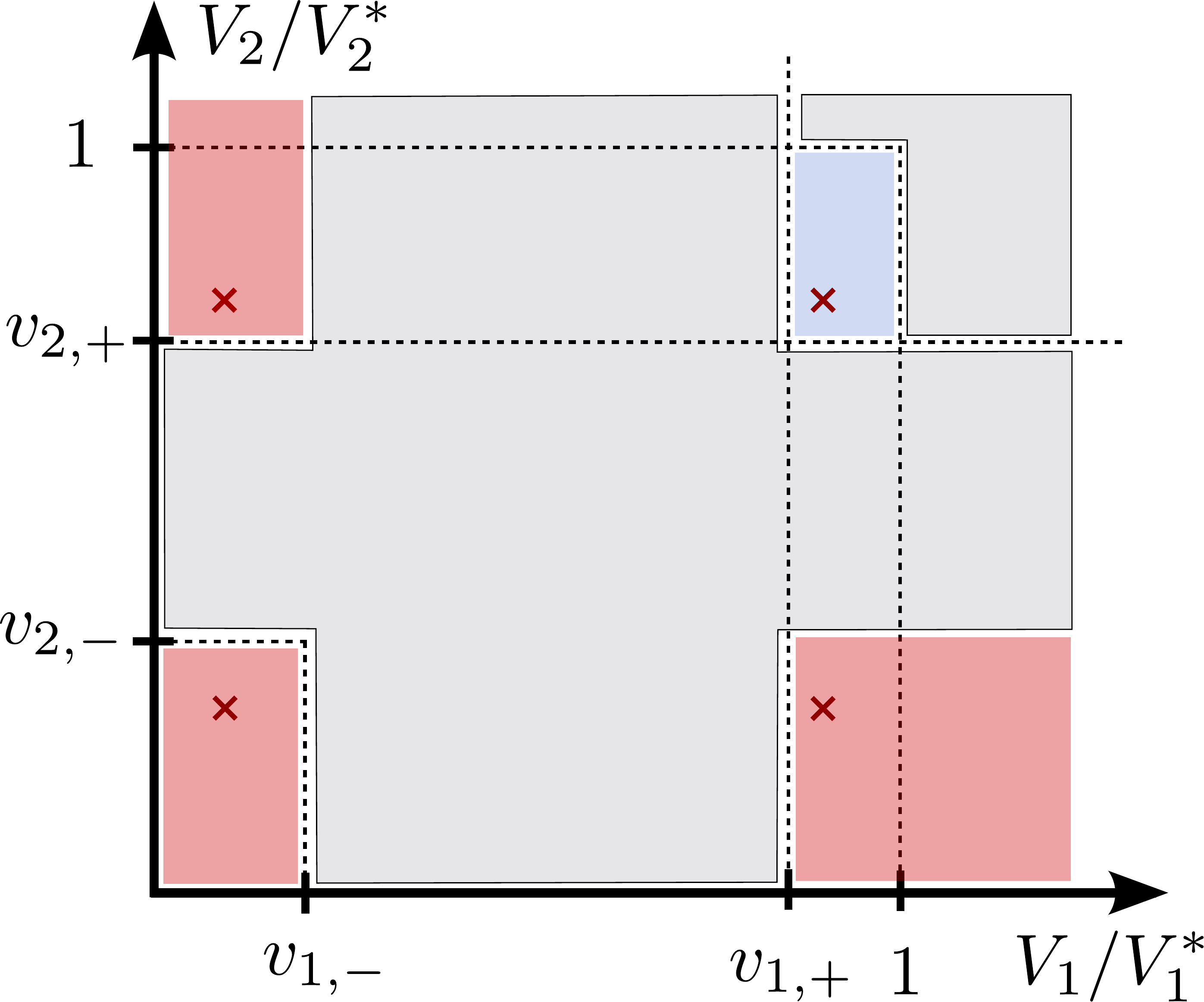}
                \caption{Space of \PQ bus voltage magnitudes for the network of Figure \ref{Fig:example-network-1b} under the conditions of Theorem \ref{Thm:Main1}. Power flow solutions are marked by red crosses (\textsf{\textcolor{red}{x}}).}
                \label{Fig:v-diagram}
\end{center}
\end{figure}

\begin{proof}.
The proof has four main steps: 1) showing that under the conditions \eqref{Eq:Omega1}--\eqref{Eq:Gamma1}, the quantities $v_{i,\pm}, \gamma_{i}$ and $\gamma_{ij}$ in \eqref{Eq:delta1}--\eqref{Eq:gamma1} are well-defined and belong to the specified sets, 2) showing existence, 3) showing uniqueness, and 4) showing necessity.
\textbf{Step 1:} Under \eqref{Eq:Omega1}, $\Delta_i$ and $\Gamma_{i}$ are constrained to the semi-open parameter set 
$$
\mathcal{P}_i \triangleq \setdef{(\Delta_i,\Gamma_{i})\in\real^2_{\geq 0}}{\Delta_i + 4\Gamma_{i}^2 < 1}\,,
$$
shown in Figure \ref{Fig:P1set}. For each pair $(\Delta_i,\Gamma_{i}) \in \mathcal{P}_i$ the inner and outer roots in \eqref{Eq:delta1} are well-posed, and hence $v_{i,\pm}$ are well-defined. Moreover, for $(\Delta_i,\Gamma_{i})\in\overline{\mathcal{P}}_i$ inspection shows that $v_{i,+}$ is a strictly decreasing function of both $\Delta_i$ and $\Gamma_{i}$, achieving its maximum of $1$ at $(\Delta_i,\Gamma_{i}) = (0,0) \in \mathcal{P}_i$ and its infimum of $\frac{1}{2}$ at $(\Delta_i,\Gamma_{i}) = (1,0) \in \overline{\mathcal{P}}_i$. Hence $v_{i,+} \in (\frac{1}{2},1]$ as claimed. Similarly, $v_{i,-}$ is a strictly increasing function of both $\Delta_i$ and $\Gamma_{i}$, achieving its minimum of $0$ at $(\Delta_i,\Gamma_{i})=(0,0) \in \mathcal{P}_i$ and its supremum of $\frac{1}{\sqrt{2}}$ at $(\Delta_i,\Gamma_{i}) = (0,\frac{1}{2}) \in \overline{\mathcal{P}}_i$. Hence $v_{i,-} \in {[0,\frac{1}{\sqrt{2}})}$ as claimed.
To establish that $\gamma_{i}$ is well-defined, square \eqref{Eq:gamma1}:
\begin{equation*}\label{Eq:Gammaglwelldefined}
\sin^2(\gamma_{i}) = \frac{\Gamma_{i}^2}{(v_{i,+})^2} = \frac{2\Gamma_{i}^2}{{1-\frac{\Delta_i}{2} + \sqrt{1-(\Delta_i+4\Gamma_{i}^2)}}}\,,
\end{equation*}
where we have substituted for $v_{i,+}$ using \eqref{Eq:delta1}. One quickly computes from this equality that
$$
\sup_{(\Delta_i,\Gamma_{i})\in\mathcal{P}_i}\sin^2(\gamma_{i}) = 1/2
$$
which is uniquely achieved at $(\Delta_i,\Gamma_{i}) = (0,\frac{1}{2}) \in \overline{\mathcal{P}_i}$. This establishes that $\sin(\gamma_{i}) \in {[0,\frac{1}{\sqrt{2}})}$ and therefore that $\gamma_{i} \in {[0,\frac{\pi}{4})}$ as claimed. That $\gamma_{ij} \in {[0,\frac{\pi}{2})}$ for each $(i,j)\in\mathcal{E}^{gg}$ follows immediately from \eqref{Eq:Gamma1}. Thus all quantities in  \eqref{Eq:delta1}--\eqref{Eq:gamma1} are well-defined and belong to the specified sets. 

\textbf{Step 2:} We begin with the fixed point equation \eqref{Eq:FixedPoint}, and make the change of variables $x = v - \vones[n]$, which shifts the open-circuit solution to the origin. Under the assumption that there are no branches between \PQ buses, we have that $\mathcal{E}^{\ell\ell} = \emptyset$, $\mathsf{S}$ is diagonal, and the fixed point function $f$ in equation \eqref{Eq:FixedPoint} simplifies to
\begin{align}
\nonumber
f(x) &= -\frac{1}{4}\mathsf{S}^{-1}\left([Q_L]r(x) + |A|_L^{g\ell}\mathsf{D}_{g\ell}u_{g\ell}(x)\right)\,,\\
\label{Eq:FixedPointCompact}
 &= -\frac{1}{4}\mathsf{S}^{-1}[Q_L]r(x) - Nu_{g\ell}(x)\,,
\end{align}
where $r(x) = (\frac{1}{1+x_1},\ldots,\frac{1}{1+x_n})^{\sf T}$, $N \triangleq -\frac{1}{4}\mathsf{S}^{-1}|A|_L^{g\ell}\mathsf{D}_{g\ell}$, and with an abuse of notation we retain the names $f$ and $u_{g\ell}$ for the functions. In components, \eqref{Eq:FixedPointCompact} reads for each $i \in \mathcal{N}_L$ as
\begin{align}\label{Eq:FixedPointComponent}
f_i(x) &= -\frac{Q_i}{4\mathsf{S}_{ii}}\frac{1}{1+x_i} - \sum_{j\in\PV(i)}\nolimits N_{ij}u_{ji}(x)\,,
\end{align}
where $u_{ji}(x)$ is the component of $u_{g\ell}(x)$ corresponding to edge $(j,i)\in\mathcal{E}^{g\ell}$. However, one may deduce from \eqref{Eq:ugl} and \eqref{Eq:VComponentForm} that in the above sum, $u_{ji}(x)$ depends only on $x_i$ for each $j \in \PV(i)$. Thus, $f_i(x) = f_i(x_i)$, and \eqref{Eq:FixedPointComponent} reads as
\begin{subequations}
\begin{align}\label{Eq:FixedPointComponent2}
f_i(x_i) &= -\frac{\Delta_i}{4}\frac{1}{1+x_i} - \sum_{j\in\PV(i)}\nolimits N_{ij}u_{ji}(x_i)\,,\\
\label{Eq:uij}
u_{ji}(x_i) &= 1 - \sqrt{1-\frac{(p_{ji}/\mathsf{D}_{ji})^2}{(1+x_i)^2}}
\end{align}
\end{subequations}
where we have inserted $\Delta_i$ from \eqref{Eq:Delta1}, and for convenience have explicitly  written out the formula for $u_{ji}(x_i)$. The fixed point equation $x = f(x)$ is therefore $n$ \emph{decoupled scalar fixed point equations} \eqref{Eq:FixedPointComponent2}, and may be studied separately; going forward we analyze the $i$th component.

For some $\delta_i \in {[0,1)}$, we will now seek to show invariance of the compact interval $[-\delta_i,0]$ under the scalar fixed point map \eqref{Eq:FixedPointComponent2}. Suppose that $x_i \in [-\delta_i,0]$. Since $(S_{LL})_{ii} < 0$ by construction, and $Q_i \leq 0$ by assumption, it follows that $\Delta_i \geq 0$, and the first term in \eqref{Eq:FixedPointComponent2} is nonpositive. Lemma \ref{Lem:RowStochastic} shows that $N$ is row-stochastic, and thus $N_{ij}$ are nonnegative numbers. Assuming that $u_{ji}(x_i)$ is well-defined for each $j \in \PV(i)$, it follows that $f_i(x_i) \leq 0$. In fact, the previous argument is valid for any $x_i \geq -\delta_i$. It follows that if $x_i > -\delta_i$, then $f_i(x_i) \leq 0$. We conclude that there are no fixed points of $f_i$ in $\real_{>0}$.

Having established the upper bound, we now proceed to lower bound $f_i(x_i)$ as
\begin{equation}\label{Eq:Fcomponentbounding}
\begin{aligned}
f_i(x_i) &\geq -\frac{\Delta_i}{4}\frac{1}{1-\delta_i} - \sum_{j\in\PV(i)}\nolimits N_{ij}u_{ji}(x_i)\\
&\geq -\frac{\Delta_i}{4}\frac{1}{1-\delta_i} - \max_{j\in\PV(i)}u_{ji}(x_i)\cdot \underbrace{\sum_{j\in\PV(i)}\nolimits N_{ij}}_{=1}\\
%
\end{aligned}
\end{equation}
where we have used that $\sum_{j\in\PV(i)}N_{ij} = \sum_{j=n+1}^{n+m}N_{ij} = 1$ since $N$ is row-stochastic. From \eqref{Eq:uij}, it follows that
\begin{equation}\label{Eq:uijbounding}
\begin{aligned}
\max_{j\in\PV(i)}u_{ji}(x_i) &= 1 - \sqrt{1-\max_{j\in\PV(i)}\frac{(p_{ji}/\mathsf{D}_{ji})^2}{(1+x_i)^2}}\\
&\geq 1 - \sqrt{1-{\Gamma_i^2}/{(1-\delta_i)^2}}
\end{aligned}
\end{equation}
where we have inserted $\Gamma_i$ from \eqref{Eq:Gamma1gl} and used that $x_i \in [-\delta_{i},0]$. Putting things together now, we have that
$$
f_i(x_i) \geq -\frac{\Delta_i}{4}\frac{1}{1-\delta_i}  - 1 + \sqrt{1-{\Gamma_i^2}/{(1-\delta_i)^2}}\,.
$$
To ensure that $f_i(x_i) \geq -\delta_i$, we therefore require that the above is further lower-bounded by $-\delta_i$, yielding the inequality
\begin{equation}\label{Eq:ScalarFixedPointInequality}
\frac{\Delta_i}{4}\frac{1}{1-\delta_i} + 1 - \sqrt{1-\frac{\Gamma_{i}^2}{(1-\delta_i)^2}} \leq \delta_i\,.
\end{equation}
Isolating the rooted term, squaring both sides, and simplifying, we arrive at the equivalent inequality 
%
\begin{equation}\label{Eq:DeltaTempe}
(1-\delta_i)^4 - (1-\delta_i)^2\left(1-\frac{\Delta_i}{2}\right) + \Gamma_{i}^2 + \frac{1}{16}\Delta_i^2  \leq 0\,.
\end{equation}
Applying Lemma \ref{Lem:QuarticI}, we find that there exists an interval of values for $\delta_i$ satisfying \eqref{Eq:DeltaTempe} if and only if $\Delta_i + 4\Gamma_i^2 < 1$. In particular, the largest such interval is given by $\mathcal{I}_i = [\delta_{i,-},\delta_{i,+}]$ where $\delta_{i,\pm}$ are defined as $$
\delta_{i,\pm} = 1-v_{i,\mp}\,,
$$
with $v_{i,\mp}$ as in \eqref{Eq:delta1}. The inequality \eqref{Eq:DeltaTempe} is satisfied with strict inequality sign for $\delta_i \in \mathrm{int}(\mathcal{I}_i) = (\delta_{i,-},\delta_{i,+})$, and with strict equality sign for $\delta_i \in \{\delta_{i,-},\delta_{i,+}\}$. We have therefore established that if $\Delta_i + 4\Gamma_i^2 < 1$, then for any $\delta_i \in [\delta_{i,-},\delta_{i,+}]$ it holds that
\begin{equation}\label{Eq:Invariance1}
x_i \in [-\delta_i,0] \quad \Longrightarrow \quad f_i(x_i) \in [-\delta_i,0]\,,
\end{equation}
which in particular shows that the set $[-\delta_{i,-},0]$ is an invariant set for $f_i$. 
Applying the Brouwer Fixed Point Theorem \cite[Section 7, Corollary 8]{EHS:94} therefore establishes the existence of at least one fixed point $x_i \in [-\delta_{i,-},0]$ for \eqref{Eq:FixedPointComponent2}. The above results hold for all fixed point equations $x_i = f_i(x_i)$ simultaneously if and only if $\Delta_i + 4\Gamma_i^2 < 1$ for all $i \in \mathcal{N}_L$, which holds if and only if the voltage stability condition \eqref{Eq:Omega1} holds. This establishes the existence of a fixed point $x \in \real^n$ for the vector fixed point equation \eqref{Eq:FixedPointCompact} satisfying the bounds $x_i \in [-\delta_{i,-},0]$ for each $i \in \mathcal{N}_L$.
In addition, we have shown that for any $\delta_{i} \in (\delta_{i,-},\delta_{i,+})$
\begin{equation}\label{Eq:Invariance2}
x_i \in [-\delta_i,0] \quad \Longrightarrow \quad f_i(x_i) \in (-\delta_i,0]\,,
\end{equation}
meaning that $x_i$ is mapped inside the original set. It follows that $[-\delta_{i,-},0]$ is the largest invariant set contained inside $(-\delta_{i,+},0]$, and that there can be \emph{no fixed points} in the set $(-\delta_{i,+},-\delta_{i,-},)$.

\textbf{Step 3:} We now show uniqueness of the fixed point. We calculate from \eqref{Eq:FixedPointComponent2} that
\begin{equation}\label{Eq:FixedPointDerivative}
\frac{\mathrm{d} f_i}{\mathrm{d} x_i}(x_i) = \frac{\Delta_i}{4}\frac{1}{(1+x_i)^2} - \sum_{j\in\PV(i)}\nolimits N_{ij}\frac{\mathrm{d} u_{ji}}{\mathrm{d} x_i}(x_i)\,.
\end{equation}
A computation using \eqref{Eq:uij} shows that
\begin{equation}\label{Eq:Diffugle}
\frac{\mathrm{d} u_{ji}}{\mathrm{d} x_i}(x_i) = \frac{-(p_{ji}/\mathsf{D}_{ji})^2}{(1+x_i)^3}\frac{1}{\sqrt{1-(p_{ji}/\mathsf{D}_{ji})^2/{(1+x_i)^2}}}\,
\end{equation}
By similar bounding as in the existence proof above, one may deduce that that the derivative \eqref{Eq:Diffugle} is continuous under condition \eqref{Eq:Omega1} on the interval $[-\delta_{i,-},0]$, and that the derivative is nonpositive. For $x_i \in [-\delta_{i,-},0]$ we therefore have that
\begin{align*}
\max_{j\in\PV(i)}\left|\frac{\mathrm{d} u_{ji}}{\mathrm{d} x_i}\right| 
&\leq \frac{\Gamma_{i}^2}{(1-\delta_{i,-})^3}\frac{1}{\sqrt{1-\Gamma_{i}^2/(1-\delta_{i,-})^2}}\,.
\end{align*}
where we have inserted $\Gamma_i$ from \eqref{Eq:Gamma1gl}. Returning now to \eqref{Eq:FixedPointDerivative}, we have for $x_i \in [-\delta_{i,-},0]$ that
\begin{align*}
\left|\frac{\mathrm{d} f_i}{\mathrm{d} x_i}(x_i)\right| &\leq \frac{\Delta_i}{4}\frac{1}{(1-\delta_{i,-})^2} + \max_{j\in\PV(i)}\left|\frac{\mathrm{d} u_{ji}}{\mathrm{d} x_i}\right|\cdot \sum_{j\in\PV(i)}N_{ij}\\
&\leq \frac{\Delta_i}{4}\frac{1}{(1-\delta_{i,-})^2} + \frac{\Gamma_{i}^2}{(1-\delta_{i,-})^3}\frac{1}{\sqrt{1-\frac{\Gamma_{i}^2}{(1-\delta_{i,-})^2}}}\,.\\
&\triangleq \beta(\Delta_i,\Gamma_{i})\,,
\end{align*}
where we have used that $Q_i \leq 0$, that $N$ is row-stochastic by Lemma \ref{Lem:RowStochastic}, and defined the result to be $\beta(\Delta_i,\Gamma_{i})$.
We now seek to show that $\beta(\Delta_i,\Gamma_{i}) < 1$ for all $(\Delta_i,\Gamma_i) \in \mathcal{P}_i$. First, we observe that $\beta(0,0) = 0$. Next, note that since $\delta_{i,-}$ is a strictly increasing continuous function of both $\Delta_{i}$ and $\Gamma_{i}$ for $(\Delta_i,\Gamma_{i})\in\mathcal{P}_i$, so is $\beta(\Delta_i,\Gamma_i)$. To establish that $\beta < 1$, it therefore suffices to check that $\beta \leq 1$ on the boundary of $\mathcal{P}_i$ given by $\mathrm{bd}(\mathcal{P}_i) = \setdef{(\Delta_i,\Gamma_{i}) \in \overline{\mathcal{P}_i}}{\Delta_i+4\Gamma_{i}^2 = 1}$. We find from \eqref{Eq:delta1} using simple algebra that
$$
1-\delta_{i,-}\big|_{\mathrm{bd}(\mathcal{P}_i)} = \sqrt{(1-\Delta_i/2)/2}\,.
$$
Inserting this into the expression for $\beta(\Delta_i,\Gamma_{i})$ and eliminating $\Gamma_{i}$ via $\Gamma_{i}^2 = \frac{1}{4}(1-\Delta_i)$, some elementary algebra shows that
\begin{align*}
\beta|_{\mathrm{bd}(\mathcal{P}_i)} &= 1\,.
\end{align*}
It follows that $\beta(\Delta_i,\Gamma_{i}) < 1$ if and only if $\Delta_i + 4\Gamma_i^2 < 1$.
%
%
Since in addition $\frac{\mathrm{d} f_i}{\mathrm{d} x_i}(x_i)$ is continuous in $x_i$ on $[-\delta_{i,-},0]$, it follows from Lemma \ref{Lem:Contraction} that $f_i$ is a contraction mapping on $[-\delta_{i,-},0]$. Finally, since $[-\delta_{i,-},0])$ is compact and (as previously shown) invariant under $f_i$, it follows from the Banach Fixed Point Theorem \cite[Theorem 9.32]{WR:76} that $f_i$ possess a unique fixed point $x_i \in [-\delta_{i,-},0])$. Combining the results component by component, we conclude that under the voltage stability condition \eqref{Eq:Omega1}, the vector fixed point equation \eqref{Eq:FixedPointCompact} possess a unique fixed point satisfying the bounds $x_i \in [-\delta_{i,-},0]$ for each $i \in \mathcal{N}_L$, or equivalently, $v_i \in [v_{i,+},1]$.
We now return to the equation \eqref{Eq:psix} for the phase angle differences $\eta$, which we write component-wise as
\begin{subequations}
\begin{align}
\sin(\eta_{ji}) &= \frac{p_{ji}/\mathsf{D}_{ji}}{v_i} = \frac{p_{ji}/\mathsf{D}_{ji}}{1+x_i}\,, &\quad (j,i) & \in \mathcal{E}^{g\ell}\,,\\\
\sin(\eta_{ji}) &= p_{ji}/\mathsf{D}_{ji}\,, &\quad (j,i) &\in \mathcal{E}^{gg}\,.
\end{align}
\end{subequations}
For $(j,i)\in\mathcal{E}^{g\ell}$ we compute that
\begin{align*}
|\sin(\eta_{ji})| &\leq \max_{j\in\PV(i)}\frac{|p_{ji}|/\mathsf{D}_{ji}}{1+x_i} = \frac{\Gamma_i}{1+x_i} \leq \frac{\Gamma_i}{1-\delta_{i,-}}\\
&= \sin(\gamma_i) < 1\,,
\end{align*}
where in the last line we have inserted \eqref{Eq:gamma1}. It follows similarly that for each $(j,i)\in\mathcal{E}^{gg}$
$$
|\sin(\eta_{ji})| \leq \sin(\gamma_{ij}) < 1\,,
$$
where $\gamma_{ij}$ is as in \eqref{Eq:gamma1}.
Therefore, by applying $\mathrm{arcsin}$ component-wise to both sides, the equation \eqref{Eq:psix} is solvable for a unique vector of angle differences $\eta$ which component-wise satisfies the bounds $|\eta_{ji}| \leq \gamma_i$ for $(j,i)\in\mathcal{E}^{g\ell}$ and $|\eta_{ij}| \leq \gamma_{ij}$ for $(i,j)\in\mathcal{E}^{gg}$.
Equivalently, by \cite[Corollary 3.3]{JWSP:17a} the active power flow \eqref{Eq:Active} is solvable for a unique angle solution $\theta \in \Theta(\pi/2)$ with angular differences $\eta = A^{\sf T}\theta$; this establishes that $u_{ij}(x)$ was in fact well-defined during the previous manipulations. Together, the above shows that the power flow equations \eqref{Eq:Active}--\eqref{Eq:Reactive} possess a unique solution $(\theta,V_L)$ satisfying the bounds \eqref{Eq:SolutionBoundsMain1}, and also establishes the first ``moreover'' statement.

\textbf{Step 4:} We proceed to necessity. First, note that when $|\PV(i)| = 1$ for all \PQ buses $i \in \mathcal{N}_L$, then the sum in \eqref{Eq:FixedPointComponent} contains only one term. In this case, no bounding is required: each decoupled fixed point equation $x_i = f_i(x_i)$ may be manipulated into the form \eqref{Eq:SingleLineIntermediate} and then directly solved for its two unique solutions, which will be well defined if and only if \eqref{Eq:Conditions1} holds; we omit the details.
Under loading scenario (i), the equations reduce to decoupled reactive power flow, and the result was shown in \cite[Supplementary Theorem 1]{jwsp-fd-fb:14c}. Under loading scenario (iii), the equations reduce to decoupled active power flow on a radial network, and the result was shown in \cite{FD-MC-FB:11v-pnas}. Necessity for loading scenario (ii) is shown by contraposition in the supplementary appendix.
%
%
%
%
%
%
%
%
%
\end{proof}

\begin{remark}[\bf Convergence of FPPF Iteration]\label{Rem:FPPFConvergence} 
The proof of uniqueness in Theorem \ref{Thm:Main1} relies on showing that the fixed-point power flow $v = f(v)$ is a contraction mapping. Moreover, we showed that the conditions under which $f$ is a contraction are also tight conditions for existence of a solution. It follows that under our conditions, the FPPF iteration $v_{k+1} = f(v_{k})$ converges to the unique high-voltage solution from any initial condition within the blue or grey boxes in Figure \ref{Fig:v-diagram}, and if it does not converge, there is likely no solution to be found. While these conclusions are restricted to the radial case with no connections between \PQ buses, this provides theoretical support for the reliability and robustness of the FPPF iteration \cite[Algorithm 1]{JWSP:17a} from Part I.
\oprocend
\end{remark}

\subsection{Main Results 2: General Radial Case}
\label{Sec:MainResults2}

In this section we allow for the more general radial case where \PQ buses are connected to one another; an example network is shown in Figure \ref{Fig:example-network-2a}. The presence of \PQ\!\!--\PQ connections (which are absent from the two-bus model of Section \ref{Sec:SolutionTwoBus}) severely complicates the analysis. The conditions we derive treat active power flows between \PQ buses quite conservatively, and can guarantee only existence.

\begin{figure}[!ht]
        \begin{center}
        \includegraphics[height=0.15\columnwidth]{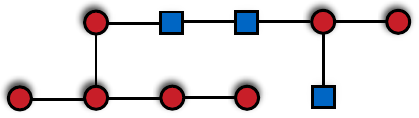}
		\caption[]{An arbitrary radial network of \PQ buses ({\tikz\draw[black,semithick,fill=red] (0,0) circle (.5ex);}) and \PV buses ({\tikz{\path[draw=black,semithick,fill={rgb:red,0;green,102;blue,196}] (0,0) rectangle (0.12cm,0.12cm);}}). The network displays \PV\!\!--\PV\!\!, \PV\!\!--\PQ\!\!, and \PQ\!\!--\PQ connections.}
        \label{Fig:example-network-2a}
\end{center}
\end{figure}

\begin{theorem}\label{Thm:Main2}\textbf{(Solvability of Lossless Power Flow on Radial Networks II)}
Consider the lossless power flow equations \eqref{Eq:Active}--\eqref{Eq:Reactive} and assume the network $(\mathcal{N},\mathcal{E})$ is radial. Let the branch and nodal stiffness matrices $\mathsf{D}$ and $\mathsf{S}$ be as in \eqref{Eq:DMatrix} and \eqref{Eq:Qcrit} respectively, with the branch-wise active power flows $p = (p_{\ell\ell},p_{g\ell},p_{gg})$ as in \eqref{Eq:BranchFlows}, and the partitioning of the unoriented incidence matrix $|A|$ as in \eqref{Eq:IncidenceAbs}. Define the maximum \PQ bus voltage stress by
\begin{equation}\label{Eq:Delta2}
\Delta \triangleq \Big\|\mathsf{S}^{-1}\left(Q_L - 4|A|_L^{\ell\ell}\mathsf{D}_{\ell\ell}^{-1}[p_{\ell\ell}]p_{\ell\ell}\right)\Big\|_{\infty}\,,
\end{equation}
and the maximum \PQ\!\!--\PQ angle stress, \PV\!\!--\PQ angle stress, and \PV\!\!--\PV angle stress by
\begin{equation}\label{Eq:Gamma2}
\begin{aligned}
\Gamma_{\ell\ell} &\triangleq \|\mathsf{D}_{\ell\ell}^{-1}p_{\ell\ell}\|_{\infty}\,, \qquad \Gamma_{g\ell} \triangleq \|\mathsf{D}_{g\ell}^{-1}p_{g\ell}\|_{\infty}\,,\\
\Gamma_{gg} &\triangleq \|\mathsf{D}_{gg}^{-1}p_{gg}\|_{\infty}\,.
\end{aligned}
\end{equation}
If the above quantities satisfy
\begin{subequations}
\begin{align}
\label{Eq:Omega2}
\Delta + 4\Gamma_{g\ell}^2 &< 1\\
\label{Eq:Gammall2}
\Gamma_{\ell\ell} &< 1/4\\
\label{Eq:Gammagg2}
\Gamma_{gg} &< 1\,,
\end{align}
\end{subequations}
then the power flow equations \eqref{Eq:Active}--\eqref{Eq:Reactive} possess a solution $(\theta,V_L)$ satisfying the bounds
\begin{subequations}\label{Eq:SolutionBoundsMain2}
\begin{align}
\label{Eq:SolutionBoundsMain21}
{\tb v_{+} \leq V_i/V_i^*} &{\tb \leq 1}\,, &\quad i &\in \mathcal{N}_L\\
\label{Eq:SolutionBoundsMain22}
|\eta_{ij}| &\leq \gamma_{\ell\ell}\,, &\quad (i,j) &\in \mathcal{E}^{\ell\ell}\\
\label{Eq:SolutionBoundsMain23}
|\eta_{ij}| &\leq \gamma_{g\ell}\,, &\quad (i,j) &\in \mathcal{E}^{g\ell}\\
\label{Eq:SolutionBoundsMain24}
|\eta_{ij}| &\leq \gamma_{gg}\,, &\quad (i,j) &\in \mathcal{E}^{gg}
\end{align}
\end{subequations}
where $v_+ \in (\frac{1}{2},1]$, and $\gamma = (\gamma_{\ell\ell},\gamma_{g\ell},\gamma_{gg}) \in {[0,\frac{\pi}{2})}\times {[0,\frac{\pi}{4})} \times {[0,\frac{\pi}{2})}$ are defined by
\begin{subequations}\label{Eq:DefinedQuantities2}
\begin{align}
\label{Eq:delta2}
v_+ &\triangleq \sqrt{\frac{1}{2}\left(1-\frac{\Delta}{2} + \sqrt{1-(\Delta+4\Gamma_{g\ell}^2)}\right)}\,,\\
\label{Eq:gamma2}
\sin(\gamma_{\ell\ell}) &\triangleq {\Gamma_{\ell\ell}}/{v_+^2}\,,\quad \sin(\gamma_{g\ell}) \triangleq {\Gamma_{g\ell}}/{v_+}\,,\\
\label{Eq:gamma22}
\sin(\gamma_{gg}) &\triangleq \Gamma_{gg}\,.
\end{align}
\end{subequations}
\end{theorem}

\begin{proof}.
Available in the supplementary appendix.
\end{proof}

Theorem \ref{Thm:Main2} gives weaker results than Theorem \ref{Thm:Main1}, guaranteeing only the existence of a solution satisfying the bounds \eqref{Eq:SolutionBoundsMain2}, but not uniqueness. This discrepancy between the two results reflects the difficulty of non-conservatively analyzing active power flows between \PQ buses (as reflected in the final two terms of the FPPF \eqref{Eq:FixedPoint}).
The quantity $\Delta$ in \eqref{Eq:Delta2} now takes into account these active power flows between \PQ buses. 
The condition \eqref{Eq:Omega2} now combines the worst-case ``reactive'' power stress \eqref{Eq:Delta2} with the worst-case active power stress $\Gamma_{g\ell}$ between \PV and \PQ buses, as opposed to the ``node-by-node'' condition \eqref{Eq:Omega1}. The bounds \eqref{Eq:SolutionBoundsMain2} on the solution are much the same as in Theorem \ref{Thm:Main1}, but are now uniform, as opposed to the line and node-specific bounds in \eqref{Eq:SolutionBoundsMain1}. In contrast to the case of Section \ref{Sec:MainResults1}, the nodal stiffness matrix $\mathsf{S}$ is now no longer diagonal, and its inverse \textemdash{} a dense, nonpositive, impedance-like matrix \textemdash{} plays a key role: the matrix-vector product in {\tb \eqref{Eq:Delta2}} combines the locations of power flows with the local strength of the network in that area, quantifying the interplay between topology and load locations.

\smallskip

\begin{remark}[\bf Conservatism of Conditions]\label{Rem:Conserve}
The condition \eqref{Eq:Omega2} strongly penalizes active power flows between \PQ buses. For example, while in simple test cases the angle differences between \PQ buses may reach as high as 25$^\circ$, the bounds generated from \eqref{Eq:gamma2} typically constrain these differences to be less than a few degrees. Active power flows between \PQ buses are also constrained by the condition {\tb \eqref{Eq:Gammall2}}; this condition is a technical requirement used in the proof, and is much more accomodating than the limits imposed by \eqref{Eq:Omega2}. Thus, the condition on $\Gamma_{\ell\ell}$ in \eqref{Eq:Gammall2} can be ignored.\footnote{The author believes that \eqref{Eq:Omega2} should in fact imply \eqref{Eq:Gammall2}, but has not been able to prove that this is the case.} \oprocend
\end{remark}

\smallskip

{\tb
\begin{remark}[\bf Comparison with Literature]\label{Rem:Comparison} The results here unify and generalize recent sufficient conditions developed for solvability of decoupled active \cite{FD-MC-FB:11v-pnas} and reactive \cite{jwsp-fd-fb:14c} power flow. In our notation, the active power flow condition in \cite{FD-MC-FB:11v-pnas} reads as $\Gamma_{gg} < 1$, while the reactive power flow condition in \cite{jwsp-fd-fb:14c} reads as $\Delta = \|{\sf S}^{-1}Q_L\|_{\infty} < 1$. As Theorem \ref{Thm:Main1} shows, these conditions are necessary for the existence of a solution to the coupled equations, but interestingly are not sufficient.

To compare and contrast our conditions with those in the literature for distribution systems, we restrict our formulation to the case with one $\PV$ bus, and restrict the distribution systems to be lossless.\footnote{These are admittedly somewhat unnatural restrictions for the respective formulations, but the comparison is nonetheless instructive.} In our notation, the feasibility conditions in \cite{SB-SZ:15,SY-HDN-KST:15} become $\|{\sf S}^{-1}\|_2^* \|P_L - \boldsymbol{\mathrm{j}}Q_L\|_{2} < 1$, where $\|A\|_2^* = \max_i(\sum_{j} |A_{ij}|^2)^{1/2}$, while the condition from \cite{CW-AB-JYLB-MP:16} becomes $\|{\sf S}^{-1}(P_L-\boldsymbol{\mathrm{j}}Q_L)\|_{\infty} < 1$, respectively. We make two observations. First, these conditions are specified in terms of bus injections of active power, while our conditions use branch flows of active power. Second, these conditions are linear in the active injections. Consequently, these conditions will never be necessary and sufficient for lossless networks, as our conditions in Theorem \ref{Thm:Main1} are quadratic in active power flows. \oprocend
\end{remark}
}

\section{Conclusions and Open Problems}
\label{Sec:Conclusions}

{\tb Here in Part II we have leveraged the FPPF model developed in Part I to derive parametric conditions under which the lossless power flow equations in radial networks are guaranteed to be solvable. We first presented a detailed analysis of the two-bus case, which motivates the network results. Our first and strongest result (Theorem \ref{Thm:Main1}) established sufficient (and tight) conditions for the existence and uniqueness of a high-voltage small-angle-difference solution. This first result is restricted to networks without direct connections between \PQ buses, and naturally generalizes all results from the two-bus case. Our second result (Theorem \ref{Thm:Main2}) eliminates this restriction, but guarantees only the existence of a high-voltage solution.}

The results here are a further step towards a deeper theoretical understanding of power flow. Several direct extensions such as incorporating voltage limits $V_{i}^{\rm min} \leq V_i \leq V_i^{\rm max}$ at \PQ buses and a priori reactive power limits $Q_{i}^{\rm min} \leq Q_i \leq Q_i^{\rm max} $ at \PV buses are feasible, but have not been pursued here. However, non-trivial problems remain unaddressed. These include
\begin{enumerate}
\item establishing nonconservative conditions under which $f$ is a contraction for general radial networks, thereby showing convergence of the FPPF iteration $v_{k+1} = f(v_k)$,
\item analysis of the meshed FPPF model from Part I,
\item the relationship between contractivity of the fixed-point power flow, convexity of the energy function \cite{KD-SL-MC:15a}, monotonicity of the power flow equations \cite{KD-SL-MC:15b}, and so-called voltage regularity \cite{PAL-DJH-SA-GA:93} and \PQ controllability \cite{TTL-RAS:91},
\item the implications of our existence conditions for the stability of grid dynamics (e.g., swing equations),
\item the relationship between our necessary and sufficient conditions and the litany of heuristic voltage stability indices proposed in the power systems literature,
\item extension of the FPPF model to lossy networks, and
\item applications of FPPF model and existence conditions; some possible avenues are ultra-fast contingency screening and distributed control.
\end{enumerate}

%
\IEEEpeerreviewmaketitle

\section*{Acknowledgments}

The author thanks D. Molzahn, F. D\"{o}rfler, K. Dvijotham, K. Turitsyn, and T. Coletta for discussions related to this work.

\ifCLASSOPTIONcaptionsoff
  \newpage
\fi


\bibliographystyle{IEEEtran}
\bibliography{alias,Main,JWSP,New}

\appendices

\section{Technical Lemmas}




\medskip


\begin{lemma}\label{Lem:RowStochastic}
The matrix
\begin{equation*}
N \triangleq -\frac{1}{4}\mathsf{S}^{-1}|A|_L^{g\ell}\mathsf{D}_{g\ell}
\end{equation*}
is row-stochastic; that is, $N_{ij} \geq 0$ and $N\vones[g\ell] = \vones[n]$.
\end{lemma}

\begin{pfof}{Lemma \ref{Lem:RowStochastic}}
First, note that since $-\mathsf{S}$ is a nonsingular $M$-matrix, $\mathsf{S}^{-1}$ has nonpositive elements, and hence from \eqref{Eq:IncidenceAbs} and \eqref{Eq:DMatrix} so does the product $\frac{1}{4}\mathsf{S}^{-1}|A|_L^{g\ell}\mathsf{D}_{g\ell}$. We compute using \eqref{Eq:Qcrit} that
\begin{align*}
-\frac{1}{4}\mathsf{S}^{-1}|A|_L^{g\ell}\mathsf{D}_{g\ell}\vones[g\ell] &= -[V_L^*]^{-1}B_{LL}^{-1}[V_L^*]^{-1}|A|_L^{g\ell}\mathsf{D}_{g\ell}\vones[g\ell]\,.
\end{align*}
Since the susceptance matrix $B$ and the incidence matrix $A$ are related by $B = A \,[b_{ij}]_{(i,j)\in\mathcal{E}} \, A^{\sf T} + [b_{\rm shunt}]$, one may use the partitioning of the incidence matrix \eqref{Eq:Incidence} to show that
\begin{align}\nonumber
B_{LG} &= A_L^{g\ell}[b_{ij}]_{(i,j)\in\mathcal{E}^{g\ell}}(A_G^{g\ell})^{\sf T}\\
\label{Eq:BLGIncidence}
&= |A|_L^{g\ell}[B_{ij}]_{(i,j)\in\mathcal{E}^{g\ell}}(|A|_G^{g\ell})^{\sf T}\,.
\end{align}
One can then verify using \eqref{Eq:IncidenceAbs}, \eqref{Eq:DMatrix} and \eqref{Eq:BLGIncidence} that
\begin{equation}\label{Eq:genloadcalculation}
\begin{aligned}
|A|_L^{g\ell}\mathsf{D}_{g\ell}\vones[g\ell] &= |A|_L^{g\ell}[V_j^*B_{ij}V_i^*]_{(i,j)\in\mathcal{E}^{g\ell}}\vones[g\ell]\\
&= [V_L^*]|A|_L^{g\ell}[B_{ij}V_i]_{(i,j)\in\mathcal{E}^{g\ell}}\vones[g\ell]\\
&= [V_L^*]|A|_L^{g\ell}[B_{ij}]_{(i,j)\in\mathcal{E}^{g\ell}}(|A|_G^{g\ell})^{\sf T}V_G\\
&= [V_L^*]B_{LG}V_G\,.
\end{aligned}
\end{equation}
Inserting $V_L^*$ from \eqref{Eq:VLstar}, it follows that
\begin{align*}
N\vones[g\ell] &= -\frac{1}{4}\mathsf{S}^{-1}|A|_L^{g\ell}\mathsf{D}_{g\ell}\vones[g\ell]\\ &= -[V_L^*]^{-1}B_{LL}^{-1}B_{LG}V_G = [V_L^*]^{-1}V_L^* = \vones[n]\,,
\end{align*}
which completes the proof.
\end{pfof}


\begin{lemma}[\bf Solutions of a Quartic Inequality]\label{Lem:QuarticI}
For $\delta \in {[0,1)}$, consider the quartic inequality
\begin{equation}\label{Eq:QuarticI}
(1-\delta)^4 - (1-\delta)^2\left(1-\frac{\Delta}{2}\right) + \Gamma^2 + \frac{1}{16}\Delta^2 \leq 0\,,
\end{equation}
with parameters $\Delta, \Gamma \geq 0$. The following two statements are equivalent:
\begin{enumerate}[(i)]\setlength{\itemsep}{1.5pt}
\item $\Delta + 4\Gamma^2 < 1$\,;
\item the inequality \eqref{Eq:QuarticI} is satisfied on the interval $\mathcal{I} = [\delta_-,\delta_+] \subset [0,1]$ and satisfied with strict inequality on the interior $\mathrm{int}(\mathcal{I})$, where $\delta_- \in {[0,\frac{1}{2})}$ and $\delta_+ \in {(1-\frac{1}{\sqrt{2}},1]}$ are the unique solutions to \eqref{Eq:QuarticI} with equality sign, satisfying $0 \leq \delta_- < \delta_+ \leq 1$ and defined by
$$
\delta_{\mp} = 1 - \sqrt{\frac{1}{2}\left(1-\frac{\Delta}{2}\pm\sqrt{1-\left(\Delta + 4\Gamma^2\right)}\right)}\,.
$$
\end{enumerate}
\end{lemma}

\begin{proof}.
The proof follows by directly solving \eqref{Eq:QuarticI} with equality sign and examining the properties of the solutions.
\end{proof}

\begin{lemma}[\bf Jacobian Conditions for Contraction]\label{Lem:Contraction}
Let $\mathcal{X} \subset \real^n$ be convex set and let $\map{f}{\mathcal{X}}{\real^n}$ be a $C^1$ function on $\mathcal{X}$. If there exists a $\beta \in [0,1)$ such that $\|\frac{\partial f}{\partial x}(x)\| \leq \beta$ for all $x \in \mathcal{X}$, then $f$ is a contraction mapping on $\mathcal{X}$. 
\end{lemma}

%

\vspace{-2em}

\begin{IEEEbiography}[{\includegraphics[width=1in,height=1.25in,clip,keepaspectratio]{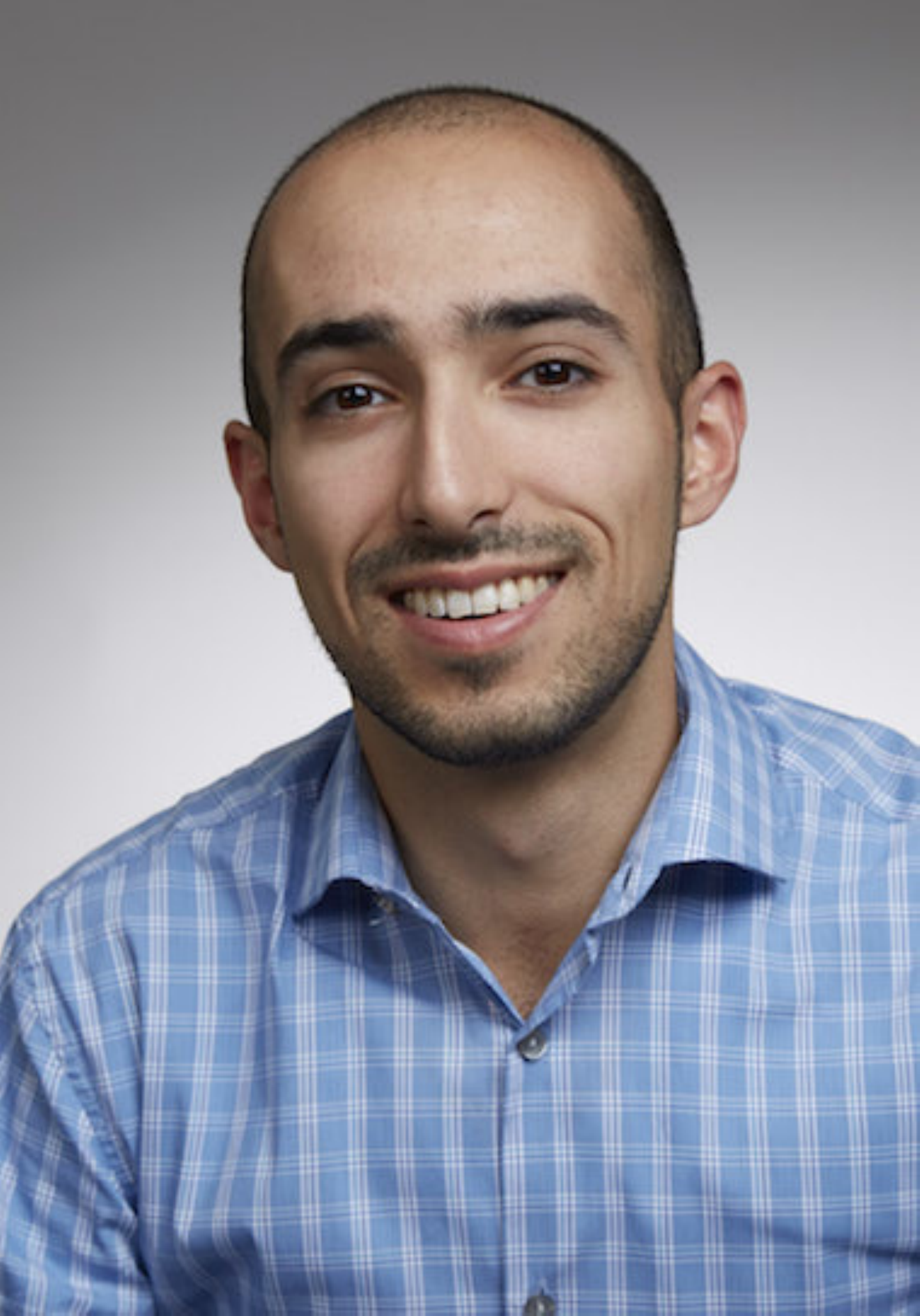}}]{John W. Simpson-Porco} (S'11--M'16) received the B.Sc. degree in engineering physics from Queen's University, Kingston, ON, Canada in 2010, and the Ph.D. degree in mechanical engineering from the University of California at Santa Barbara, Santa Barbara, CA, USA in 2015.

He is currently an Assistant Professor of Electrical and Computer Engineering at the University of Waterloo, Waterloo, ON, Canada. He was previously a visiting scientist with the Automatic Control Laboratory at ETH Z\"{u}rich, Z\"{u}rich, Switzerland. His research focuses on the control and optimization of multi-agent systems and networks, with applications in modernized power grids.

Prof. Simpson-Porco is a recipient of the 2012--2014 IFAC Automatica Prize and the Center for Control, Dynamical Systems and Computation Best Thesis Award and Outstanding Scholar Fellowship.
\end{IEEEbiography}
%
%
%

\section{Omitted Proofs}




\begin{pfof}{Theorem \ref{Thm:Main1} (Continued)}
Here we complete the proof that for loading scenario (ii) in the theorem statement, the conditions \eqref{Eq:Conditions1} are necessary and sufficient as a function of $\alpha$. We proceed by contraposition and show that this load profile yields a case where the voltage stability condition \eqref{Eq:Omega1} is satisfied with equality and the conclusions of the theorem do not hold. Since $Q_i = 0$ for each \PQ bus $i \in \mathcal{N}_L$, it follows that $\Delta_i = 0$. By construction, for the scenario described the branch active power flows are given by $p = (p_{g\ell},p_{gg})^{\sf T}$ where $p_{gg} = \vzeros[gg]$ and $p_{g\ell} = \frac{\alpha}{2}\mathsf{D}_{g\ell}\vones[g\ell]$.  Therefore, we calculate for each $i \in \mathcal{N}_{L}$ that $\Gamma_{i} = \alpha/2$, and for each $(i,j)\in\Egg$ that $\Gamma_{ij} = 0$. The voltage stability condition \eqref{Eq:Omega1} becomes $\alpha^2 < 1$, and \eqref{Eq:delta1} becomes $v_{i,\pm}(\alpha) = v_{\pm}(\alpha) \triangleq \sqrt{(1 \pm \sqrt{1-\alpha^2})/2}$. We claim that two fixed points of \eqref{Eq:FixedPointCompact} are given explicitly by $x_{\pm}(\alpha) = -\delta_{\pm}(\alpha)\vones[n]$, where $\delta_{\pm}(\alpha) = 1-v_{\mp}(\alpha)$. To see this, note first from \eqref{Eq:ugl} that
\begin{align*}
u_{g\ell}(x_{\pm}) &= 
\left(1-\sqrt{1-{\alpha^2}/{4(1-\delta_{\pm})^2}}\right)\vones[g\ell]
\end{align*}
and hence that
$$
f(x_{\pm}) = Nu_{g\ell}(x_{\pm}) = \left(1-\sqrt{1-{\alpha^2}/{4(1-\delta_{\pm})^2}}\right)\vones[n]\,,
$$
since $N$ is row-stochastic. A somewhat messy calculation confirms that the above reduces to $f(x_{\pm}) = x_{\pm}$, so $x_{\pm}$ are both in fact fixed points. 
It follows that for $\alpha \in [0,1)$, $x_-(\alpha) \in [-\delta_-,0]^n$ is the unique fixed point described by the theorem. However, as $\alpha \rightarrow 1$ the solutions $x_-(\alpha)$ and $x_+(\alpha)$ coalesce at the point $x = (1-\frac{1}{\sqrt{2}})\vones[n]$ in a saddle-node bifurcation.
By continuity then, at $\alpha = 1$ when the condition \eqref{Eq:Omega1} is only marginally satisfied, there are \emph{two} solutions within the set $[-\delta_-,0]^n$, both of which occur on the boundary of the set. We therefore have non-uniqueness of the solution in this set, which shows the statement of the theorem fails, completing the proof.
\end{pfof}


\begin{pfof}{Theorem \ref{Thm:Main2}}
That $v_+$, $\gamma_{g\ell}$, and $\gamma_{gg}$ are well-defined under the conditions \eqref{Eq:Omega2} and \eqref{Eq:Gammagg2} follows from arguments identical to those in the proof of Theorem \ref{Thm:Main1}. To show that $\gamma_{\ell\ell}$ is well-defined, we compute using \eqref{Eq:gamma2} that
$$
\sin(\gamma_{\ell\ell}) = \frac{\Gamma_{\ell\ell}}{(v_+)^2} < \frac{1/4}{(\frac{1}{2})^2} = 1\,,
$$
and therefore $\gamma_{\ell\ell} \in {[0,\frac{\pi}{2})}$ as claimed. All quantities in  \eqref{Eq:delta2}--\eqref{Eq:gamma22} are therefore well-defined and belong to the specified sets. 
Changing variables to $x \triangleq v - \vones[n]$, the FPPF \eqref{Eq:FixedPoint} becomes
\begin{equation}\label{Eq:FixedPointx}
\begin{aligned}
f(x) &\triangleq -\frac{1}{4}\mathsf{S}^{-1}[Q_L]r(x) -N u_{g\ell}(x)\\
&\quad + \frac{1}{4}\mathsf{S}^{-1}A_L^{\ell\ell}(+)\,[A_L^{\ell\ell}(-)^{\sf T}(\vones[n]+x)]\mathsf{D}_{\ell\ell}u_{\ell\ell}(x)\\
&\quad + \frac{1}{4}\mathsf{S}^{-1}A_L^{\ell\ell}(-)\,[A_L^{\ell\ell}(+)^{\sf T}(\vones[n]+x)]\mathsf{D}_{\ell\ell}u_{\ell\ell}(x)\,,
\end{aligned}
\end{equation}
where $r(x)$ and $N$ are as defined below \eqref{Eq:FixedPointCompact}. For $\delta \in {[0,\frac{1}{2})}$, we will now seek to show invariance of the compact, convex set $\mathcal{B}{([-\delta,0])}\triangleq \setdef{x \in \real^n}{-\delta\vones[n]\leq x\leq\vzeros[n]}$ under the fixed-point map \eqref{Eq:FixedPointx}.
Suppose that $x \in \mathcal{B}([-\delta,0])$. The first two terms in \eqref{Eq:FixedPointx} are nonpositive by the same arguments as in the proof of Theorem \ref{Thm:Main1}. The third and fourth terms in \eqref{Eq:FixedPointx} are products on nonnegative matrices and vectors with the nonpositive matrix $\mathsf{S}^{-1}$. It follows that $f(x) \leq \vzeros[n]$.

We now proceed to lower bound $f(x)$. Working on the first term, by definition of $r(x)$ it holds that
\begin{align}\label{Eq:FirstTermBoundTemp}
\mathsf{S}^{-1}[Q_L]r(x) &\leq \mathsf{S}^{-1}[Q_L]\vones[n]\cdot \frac{1}{1-\delta} = \frac{1}{1-\delta}\mathsf{S}^{-1}Q_L
\end{align}
where we have used that $\mathsf{S}^{-1}$ and $Q_L$ both have nonpositive elements. Working on the second term now, a bounding very similar to \eqref{Eq:Fcomponentbounding}-\eqref{Eq:uijbounding} shows that
\begin{equation}\label{Eq:BoundGammaTemp2}
Nu_{g\ell}(x) \leq \left(1 - \sqrt{1-\Gamma_{g\ell}^2/(1-\delta)^2}\right)\vones[g\ell]\,,
\end{equation}
where $\Gamma_{gl}$ is as in \eqref{Eq:Gamma2}. We now direct our attention to the final two terms in \eqref{Eq:FixedPointx}. Since for any $y \in [0,1]$ it holds that $1-\sqrt{1-y} \leq y$, it follows from \eqref{Eq:ull} that
\begin{equation}\label{Eq:QuadraticBoundsForUll}
u_{\ell\ell}(v) \leq [h_{\ell\ell}(v)]^{-2}\mathsf{D}_{\ell\ell}^{-2}[p_{\ell\ell}]p_{\ell\ell}\,.
\end{equation}
Moreover, in vector form one may deduce from \eqref{Eq:IncidencePlusMinus} that $h_{\ell\ell}(v) = [A_L^{\ell\ell}(+)^{\sf T}v]\,A_L^{\ell\ell}(-)^{\sf T}v$.
We may therefore bound the third term in \eqref{Eq:FixedPointx} as
\begin{equation}\label{Eq:BadTermBound1}
\begin{aligned}
&A_{L}^{\ell\ell}(+)[A_L^{\ell\ell}(-)^{\sf T}v]\mathsf{D}_{\ell\ell}u_{\ell\ell}(v)\\
& \leq A_{L}^{\ell\ell}(+)[A_L^{\ell\ell}(-)^{\sf T}v]\mathsf{D}_{\ell\ell}[h_{\ell\ell}(v)]^{-2}\mathsf{D}_{\ell\ell}^{-2}[p_{\ell\ell}]p_{\ell\ell}\\
&= A_{L}^{\ell\ell}(+)[A_L^{\ell\ell}(-)^{\sf T}v][h_{\ell\ell}(v)]^{-2}\mathsf{D}_{\ell\ell}^{-1}[p_{\ell\ell}]p_{\ell\ell}\\
&= A_{L}^{\ell\ell}(+)[A_L^{\ell\ell}(-)^{\sf T}v][A_L^{\ell\ell}(+)^{\sf T}v]^{-2}[A_L^{\ell\ell}(-)^{\sf T}v]^{-2}\mathsf{D}_{\ell\ell}^{-1}[p_{\ell\ell}]p_{\ell\ell}\\
&= A_{L}^{\ell\ell}(+)[A_L^{\ell\ell}(+)^{\sf T}v]^{-2}[A_L^{\ell\ell}(-)^{\sf T}v]^{-1}\mathsf{D}_{\ell\ell}^{-1}[p_{\ell\ell}]p_{\ell\ell}\\
& < 4\,A_{L}^{\ell\ell}(+)\mathsf{D}_{\ell\ell}^{-1}[p_{\ell\ell}]^2[A_L^{\ell\ell}(+)^{\sf T}v]^{-1}\vones[\ell\ell]\\
&= \left(4\,A_{L}^{\ell\ell}(+)\mathsf{D}_{\ell\ell}^{-1}[p_{\ell\ell}]^2\,A_L^{\ell\ell}(+)^{\sf T}\right)[v]^{-1}\vones[n]\\
&= 4\,A_{L}^{\ell\ell}(+)\mathsf{D}_{\ell\ell}^{-1}[p_{\ell\ell}]^2\,A_L^{\ell\ell}(+)^{\sf T}r(x)\\
&\leq \frac{4}{1-\delta}\,A_{L}^{\ell\ell}(+)\mathsf{D}_{\ell\ell}^{-1}[p_{\ell\ell}]^2\,A_L^{\ell\ell}(+)^{\sf T}\vones[n]\\
&= \frac{4}{1-\delta}\,A_{L}^{\ell\ell}(+)\mathsf{D}_{\ell\ell}^{-1}[p_{\ell\ell}]p_{\ell\ell}
\end{aligned}
\end{equation}
where we have substituted for $u_{\ell\ell}(v)$ using \eqref{Eq:QuadraticBoundsForUll}, substituted for $h_{\ell\ell}(v)$, then used the bound that $x_i \geq -\delta > -\frac{1}{2}$, which is equivalent to $v_i > \frac{1}{2}$. We have also rearranged diagonal matrices at several points, used the nonnegativity of the terms in the product, and used the identity
\begin{align*}
[A_L^{\ell\ell}(\pm)^{\sf T}v]^{-1}\vones[\ell\ell] &= A_L^{\ell\ell}(\pm)^{\sf T}[v]^{-1}\vones[n]
\end{align*}
to simplify.
Finally, we then returned to the $x$ variables, and used the bound $x_i \geq -\delta$ along with the fact that $A_L^{\ell\ell}(+)^{\sf T}\vones[n] = \vones[\ell\ell]$. Similar bounding on the fourth term in \eqref{Eq:FixedPointx} leads to
\begin{equation}\label{Eq:BadTermBound2}
\begin{aligned}
A_{L}^{\ell\ell}(-)[A_L^{\ell\ell}(+)^{\sf T}v]\mathsf{D}_{\ell\ell}u_{\ell\ell}(v) &\leq \frac{4}{1-\delta}\,A_{L}^{\ell\ell}(-)\mathsf{D}_{\ell\ell}^{-1}[p_{\ell\ell}]p_{\ell\ell}
\end{aligned}
\end{equation}
Putting things together now by inserting the bounds \eqref{Eq:FirstTermBoundTemp}, \eqref{Eq:BoundGammaTemp2}, \eqref{Eq:BadTermBound1} and \eqref{Eq:BadTermBound2} into \eqref{Eq:FixedPointx} and using that $\mathsf{S}^{-1}$ has nonpositive elements, we have that
\begin{align}\nonumber
f(x) 
\label{Eq:VectorBoundedFixedPoint}
&\geq -\frac{1}{4}\frac{1}{1-\delta}\mathsf{S}^{-1}Q_{\rm effective}\\
&\quad -\left(1 - \sqrt{1-\Gamma_{g\ell}^2/(1-\delta)^2}\right)\vones[g\ell]
\end{align}
where we have defined the nonpositive vector
\begin{align}
Q_{\rm effective} &\triangleq Q_L - 4\left(A_L^{\ell\ell}(+)+A_{L}^{\ell\ell}(-)\right)\mathsf{D}_{\ell\ell}^{-1}[p_{\ell\ell}]p_{\ell\ell}\\
&= Q_L - 4|A|_{L}^{\ell\ell}\mathsf{D}_{\ell\ell}^{-1}[p_{\ell\ell}]p_{\ell\ell}\,.
\end{align}
A sufficient condition to have $f(x) \geq -\delta \vones[n]$ is therefore that
$$
\frac{1}{4}\frac{1}{1-\delta}\Delta + 1 - \sqrt{1-\frac{\Gamma_{g\ell}^2}{(1-\delta)^2}} \leq \delta\,,
$$
where we have inserted $\Delta = \|\mathsf{S}^{-1}Q_{\rm effective}\|_{\infty}$ from \eqref{Eq:Delta2}. This is the same form of inequality encountered in \eqref{Eq:ScalarFixedPointInequality} during the proof of Theorem \ref{Thm:Main1}. In the same way as before, we conclude from the Brouwer fixed-point theorem that if \eqref{Eq:Omega2} holds, then there exists a fixed-point of \eqref{Eq:FixedPointx} in the set $\mathcal{B}([-\delta_-,0])$, where $\delta_- = 1-v_{+}$ with $v_+$ as in \eqref{Eq:delta2}. Equivalently, there exists a fixed-point $v$ of \eqref{Eq:FixedPoint} satisfying $v_+\vones[n] \leq v \leq \vones[n]$. The solution bounds \eqref{Eq:SolutionBoundsMain21}, \eqref{Eq:SolutionBoundsMain23}, and \eqref{Eq:SolutionBoundsMain24} follow exactly as in Theorem \ref{Thm:Main1}. To show the bound \eqref{Eq:SolutionBoundsMain22}, we return to the active power flow \eqref{Eq:psix}, for which the appropriate subequation reads as
$$
\bsin(\eta_{\ell\ell}) = [h_{\ell\ell}(v)]^{-1}\mathsf{D}_{\ell\ell}^{-1}p_{\ell\ell}\,.
$$
We therefore compute that
\begin{align*}
\|\bsin(\eta_{\ell\ell})\|_{\infty} &\leq \|[h_{\ell\ell}(v)]^{-1}\|_{\infty}\|\mathsf{D}_{\ell\ell}^{-1}p_{\ell\ell}\|_{\infty}\\
&\leq \frac{\Gamma_{\ell\ell}}{(v_+)^2} = \sin(\gamma_{\ell\ell}) < 1\,.
\end{align*}
It follows then that the power flow equations \eqref{Eq:Active}--\eqref{Eq:Reactive} possess a solution $(\theta,V_L)$ satisfying the bounds \eqref{Eq:SolutionBoundsMain2}.
\end{pfof}

\end{document}